\documentclass[12pt]{article}
\usepackage[utf8]{inputenc}
\usepackage[ngerman,english]{babel}
\usepackage{amsmath}
\usepackage{amssymb}
\usepackage{amsthm}
\usepackage{mathtools}
\usepackage[T1]{fontenc}                                                      
\usepackage{graphicx}
\usepackage{lmodern}
\usepackage{stmaryrd} 
\usepackage{color}
\usepackage{verbatim}
\usepackage{enumerate}
\usepackage{xifthen}
\usepackage{calc}
\usepackage[hidelinks]{hyperref}
\usepackage[version=3,arrows=pgf-filled]{mhchem}
\usepackage{csquotes}
\usepackage{longtable}
\usepackage{hhline}
\usepackage{colortbl}
\usepackage{extarrows}
\usepackage[ruled]{algorithm2e} 
\usepackage{fourier}
\usepackage{geometry}

\numberwithin{equation}{section}

\DeclareMathAlphabet{\pazocal}{OMS}{zplm}{m}{n}
\newcommand{\calO}{\pazocal{O}}

\newcommand{\calR}{\pazocal{R}}
\newcommand{\calW}{\pazocal{W}}

\allowdisplaybreaks

\usepackage{listings}
\usepackage{listingsutf8}
\makeatletter
\renewcommand\part{%
  \if@openright
    \cleardoublepage
  \else
    \clearpage
  \fi
  \thispagestyle{empty}%
  \if@twocolumn
    \onecolumn
    \@tempswatrue
  \else
    \@tempswafalse
  \fi
  \null\vfil
  \secdef\@part\@spart}
\makeatother

\newtheoremstyle{poeschelstyle}
{10pt} 
{10pt} 
{\addtolength{\leftskip}{0.5cm}\itshape} 
{-0.5cm} 
{\bfseries} 
{}
{ } 
{\settowidth{\numlength}{\thmnumber{#2}\quad}\hspace*{-\the\numlength}\thmnumber{#2}\quad\ifthenelse{\isempty{#3}}{\thmname{#1}}{\thmname{#3}} \hspace{1em}}

\newtheorem{defi}{Definition}[section]
\newtheorem{theorem}[defi]{Theorem}
\newtheorem{lem}[defi]{Lemma}

\newcommand{\N}{\ensuremath{\mathbb{N}}}
\newcommand{\Z}{\ensuremath{\mathbb{Z}}}

\newcommand{\R}{\ensuremath{\mathbb{R}}}

\newcommand{\norm}[2][]{\ensuremath{\left\|#2\right\|_{#1}}} 

\newcommand{\snorm}[2][]{\ensuremath{|#2|_{#1}}}
\newcommand{\snormlr}[2][]{\ensuremath{\left|#2\right|_{#1}}}
\newcommand{\snormleftright}[2][]{\ensuremath{\left|#2\right|_{#1}}}
\newcommand{\scalarprod}[2]{\ensuremath{\left\langle{#1,#2}\right\rangle}} 

\newcommand{\kdummy}{\ensuremath{\kappa^\text{dummy}}}
\newcommand{\lambdam}{\ensuremath{\lambda_\text{met}}}
\newcommand{\lambdac}{\ensuremath{\lambda_\text{cer}}}
\newcommand{\kappam}{\ensuremath{\kappa_\text{met}}}
\newcommand{\kappac}{\ensuremath{\kappa_\text{cer}}}

\renewcommand{\det}{\ensuremath{\operatorname{det}}}
\renewcommand{\div}{\ensuremath{\operatorname{div}}}
\newcommand{\trace}{\ensuremath{\operatorname{tr}}}

\newcommand{\supp}{\ensuremath{\operatorname{supp}}}
\newcommand{\lpert}{\ensuremath{\lambda_{\text{pert}}}}
\newcommand{\mpert}{\ensuremath{\mu_{\text{pert}}}}
\newcommand{\uapprox}{\ensuremath{u_\text{approx}}}

\newcommand{\ldummy}{\ensuremath{\lambda^\text{dummy}}}
\newcommand{\lepsdummy}{\ensuremath{\lambda^{\varepsilon,\text{dummy}}}}

\makeatletter
\newcommand{\pushright}[1]{\ifmeasuring@#1\else\omit\hfill$\displaystyle#1$\fi\ignorespaces}
\newcommand{\pushleft}[1]{\ifmeasuring@#1\else\omit$\displaystyle#1$\hfill\fi\ignorespaces}
\makeatother

\begin{document}
\newlength{\numlength}	

\title{Analysis of the embedded cell method for the numerical homogenization of metal-ceramic composite materials}
\author{\renewcommand{\thefootnote}{\arabic{footnote}} Wolf-Patrick D\"ull$^{1}$, Bastian Hilder$^{1}$, Guido Schneider$^{1}$}
\date{\today}

\footnotetext[1]{Institut f\"ur Analysis, Dynamik und  Modellierung, 
       Universit\"at Stuttgart, Pfaffenwaldring 57, 70569 Stuttgart, Germany,
duell@mathematik.uni-stuttgart.de, bastian.hilder@gmx.de, guido.schneider@mathematik.uni-stuttgart.de}

\maketitle

\begin{abstract}
In this paper, we analyze the embedding cell method, an algorithm which has been developed 
for the numerical homogenization of metal-ceramic composite materials. We show the convergence of the iteration scheme of this algorithm and the coincidence of the material properties predicted by the limit with the effective material properties provided by the analytical homogenization theory in three situations, namely for a one dimensional linear elasticity model,
a simple one dimensional plasticity model and a two dimensional model of linear hyperelastic isotropic materials with constant shear modulus and slightly varying first Lamé parameter.
\\[3mm]
{\bf Keywords\;} linear elasticity, composite materials, numerical homogenization, error analysis
\\[3mm]
{\bf Mathematics Subject Classification (2010)\;} 74Bxx, 74Qxx, 74Sxx, 35Q74
\end{abstract}


\section{Introduction}

A fundamental goal of material science is to design materials with exactly prescribed properties,
which are optimal for a specific purpose. In principle, this issue could be addressed experimentally with the help of trial and error strategies. However, such strategies are often very expensive and time-consuming. Therefore, the development of computational simulations for
optimizing materials has become an emerging field of research in the recent years. The success of computational material design is boosted by the fact that computation speed and efficiency of hardware and high performance systems have strongly increased over the last decades.
Nevertheless, physical experiments still have to remain an integral part of the process 
of designing optimized materials, but the number of such experiments can be significantly reduced if the used computational methods are proven to be reliable. Therefore, a good mathematical understanding of the computational methods is very important.

The field of numerical mathematics provides well-understood numerical schemes
like the finite element method for the numerical solution of the partial differential equations modeling the behavior of the materials. 
However, many materials relevant for technical applications, such as composite materials,  have a microstructure and hence their material parameters vary on small length scales. To resolve the microstructure with the help of the classical finite element method, the mesh grid size has to be chosen very small such that the computational effort is very high. But since effects of microstructures often average out on macroscopic length scales, it is sufficient in many situations to compute the effective material behavior on macroscopic scales.  

An important mathematical strategy to determine the effective material behavior 
is the application of so-called homogenization processes. 
The general idea of homogenization is to replace a multiscale problem by a homogeneous problem with the same effective properties. In analytical homogenization theory, the homogeneous problem is obtained as the limit of a sequence of multiscale problems with faster and faster oscillating parameters, and formulas for the homogeneous parameters of the limit problem are provided, see, for example, \cite{CD99}.
 However, in general, it is difficult to compute the values of the homogeneous parameters explicitly. Therefore, it is necessary to develop numerical homogenization processes.

Such a numerical homogenization process is given by the so-called embedded cell method, an algorithm which was developed by Dong and Schmauder \cite{Schmauder96} in 1996 to compute the stress-strain curves for metal-ceramic composite materials or more generally for particle or fiber reinforced materials.
These materials are of great importance for the automotive industry, e.g. for rotor brakes, the  aerospace technology, e.g. for lightweight elements, or the medical technology, e.g. for implants. The basic idea of the embedded cell method is to replace the metal-ceramic composite material with its complex geometry by a so-called embedded cell consisting of a connected component of ceramic particles surrounded by 
 a metal matrix (or vice versa) which is embedded into a dummy material whose material parameters are determined by a self-consistent numerical iteration scheme. The volume ratio between the 
ceramic and the metallic part of the embedded cell is the same as in the original composite
material. The volume of the dummy material is larger than the volumes of the ceramic and the metallic part such that the influence of the geometry of the exterior boundary of the dummy material on the embedded cell can be neglected. Having determined the material parameters 
of the dummy material the strain-stress curve of the composite material consisting of the embedded cell and the dummy material is computed numerically.

Numerical experiments show that the strain-stress curves obtained via the embedded cell method are often very close to the strain-stress curves of the corresponding original composite materials measured in physical experiments. A first attempt to prove analytic convergence results can be found in \cite{Salit14}.
However, the question how the results of the iteration process are related to the exact solutions of the underlying mathematical equations and to the formulas for the effective material parameters provided by analy\-tical homogenization theory remained open. 

In general, the effective material behavior of a composite material does not only depend on the volume ratio of its phases but also on the spatial arrangement of the phases. For example, a stiff  material with inclusions of a soft material behaves differently than a soft material with inclusions of a stiff material even if the volume ratios between the stiff and the soft phase are the same for both composite materials. The formulas for the effective material behavior from analy\-tical homogenization theory take into account both the volume fraction and the spatial arrangement of the phases. In contrast, the embedded cell method in its present form respects the volume fraction of the phases but only allows that one connected component of one phase is surrounded by another phase and the dummy material. Hence, it cannot be expected that 
the strain-stress curves computed by the embedded cell method always coincide with the strain-stress curves obtained with the help of analytical homogenization theory. Therefore, it is important to investigate for which kind of composite materials and which spatial arrangements of their phases both homogenization methods yield comparable results. 

It is the aim of the present paper to prove some basic results on the convergence behavior of the embedded cell method and its relation to analytical homogenization theory. For composite materials whose material parameters vary only in one direction the effective material behavior 
is completely determined by the volume ratio of its phases. Consequently, we can prove for
a one dimensional model of linear elastic materials and a simple one dimensional plasticity model by explicit calculations that the embedded cell method converges and that the material properties predicted by the limit coincide with the effective material properties provided by analytical homogenization theory.

As a first step towards a general two dimensional theory we analyze a two dimensional model of linear hyperelastic isotropic composite materials with constant shear modulus and a first Lamé parameter which varies only by a small parameter $\varepsilon$. In this case, we prove that the embedded cell method converges and that the material properties of the limit coincide with the effective material properties from analytical homo\-genization theory
at least up to an error of order $\mathcal{O}(\varepsilon^2)$ for $\varepsilon$ tending to zero since the influence of the volume ratio of the phases on the effective material behavior dominates the effects induced by the spatial arrangement of the phases. 
The assumption of a constant shear modulus has the advantage that we only need to model a tensile test for determining the effective first Lamé parameter and no shear test for determining  the effective shear modulus such that we analyze the same experimental set-up as in \cite{Schmauder96}. Nevertheless, our result and our proof can be generalized in a straightforward manner to the case when also the shear modulus varies by $\varepsilon$. The proof of our result relies on perturbation theory, complete induction using the iteration procedure of the embedded cell method and an a priori estimate from elliptic theory being valid also for weak solutions 
of our model equations, which is the appropriate notion of solution in the case of composite materials. Moreover, we use a generalized formula for the tensile force which is also applicable to non-smooth data and weak solutions of the model equations.

Since the embedded cell method has similarities with the Hashin-Shtrikman coated sphere construction, cf. \cite{Ha62},\cite{HaShtr63}, we expect that for general two and three dimensional two-phase composite materials the effective first Lamé parameter and the effective shear modulus computed by the embedded cell method and  the effective first Lamé parameter and the effective shear modulus obtained with the help of analytical homo\-genization theory share the property that they can be bounded from above and below by the Hashin-Shtrikman bounds which are the tightest bounds possible for composite moduli of two-phase composite materials, cf. \cite{Ta09}. It is subject of further research to prove that the embedded cell method has this useful property.

The plan of the paper is as follows. In section \ref{chap:derivation}, we present the physical
model which will be the basis of our analysis as well as an existence and uniqueness result for solutions of the model equations. The boundary conditions in this model are chosen 
such that the model describes a tensile test. At the end of section \ref{chap:derivation}, we  derive the generalized formula for the tensile force. In section \ref{chap:ecm}, we present the  embedding cell method. Then we prove our convergence and correctness results for the one dimensional case in section \ref{chap:correctness_1D} and for the two dimensional case in section \ref{chap:perturbation_theory}.
\\\\
\textbf{Acknowledgments:} The research is partially supported
by the Cluster of Excellence \grqq SimTech\grqq{} at the University of Stuttgart. The authors are grateful for discussions with Siegfried Schmauder and Alexander Mielke.

\section{The model}
\label{chap:derivation}

In this section, we present the physical model which we will analyze in the subsequent sections.
The model consists of the basic equations for linear hyperelastic isotropic solids applied to the experimental set-up of a tensile test.

\subsection{Derivation of the model equations}

First, we recall the basic definitions of the theory of linear elasticity and the derivation of the basic equations for linear hyperelastic isotropic solids. Our presentation mainly follows
\cite{Cialet93}. 
Consider a solid occupying the so-called reference configuration $\Omega \subset \R^3$.
The solid is exposed to an external body force $f : \Omega \rightarrow \R^3$ and an external surface force $g : \partial \Omega \rightarrow \R^3$.
These external forces impose a deformation $y : \Omega \rightarrow \R^3$ and a displacement $u := y - Id$, where $Id$ is the identical map, respectively. The deformed domain $y(\Omega)$ is called the current configuration.
For $u \in C^1(\overline{\Omega})$ let $\nabla u$ be the displacement gradient and $\nabla^s u = (\nabla u + (\nabla u)^T)/2$ the symmetric part of the displacement gradient. For physical reasons it is reasonable to postulate that this deformation is  injective and orientation-preserving, i.e. $\det \nabla y > 0$.

We consider static problems for which the stress principle of Euler and Cauchy (cf. \cite[axiom 2.2-1]{Cialet93}) holds. The stress principle of Euler and Cauchy postulates the existence of the so-called Cauchy stress vector field
\begin{align*}
	t: \overline{y(\Omega)} \times \mathbb{S} \rightarrow \R^3
\end{align*}
with $\mathbb{S} := \{x \in \R^3 : \snorm{x} = 1\}$ such that the following holds:
\begin{enumerate}
	\item For an arbitrary subdomain $A \subset y(\Omega)$ and at any point $x \in \partial A \subset \partial y(\Omega)$ where the unit outer normal vector $n \in \mathbb{S}$ to the surface $\partial A$ exists it holds
	\begin{align}
		t(x,n) = g(y^{-1}(x)).
	\end{align}
	\item The axiom of force balance holds, i.e. for any subdomain $A \subset y(\Omega)$ we have
	\begin{align}
		\int_{A} f(y^{-1}(x)) dx + \int_{\partial A} t(x,n) do = 0,
	\end{align}
	where $do$ is the surface measure on $\partial A$.
	\item The axiom of moment balance holds, i.e. for any subdomain $A \subset \Omega$ we have
	\begin{align}
		\int_{A} x \times f(y^{-1}(x)) dx + \int_{\partial A} x \times t(x,n) do = 0,
	\end{align}
	where $\times: \R^3 \times \R^3 \rightarrow \R^3$ is the cross product.
\end{enumerate}

The stress principle of Euler and Cauchy implies Cauchy's theorem (cf. \cite[theorem 2.3-1]{Cialet93}), which states the existence of a symmetric tensor, the so-called Cauchy stress tensor $\sigma : y(\Omega) \rightarrow \R^{3 \times 3}$ such that
\begin{align}
	t(x, n) = \sigma(x) n
\end{align}
for all $x \in y(\Omega), n \in \mathbb{S}$. 
Since the Cauchy stress tensor is defined on the deformed configuration, which has to be determined first, we will use another tensor, which refers to the reference configuration, which is known a priori: the so-called first Piola-Kirchhoff stress tensor $P$, which is related to the Cauchy stress tensor $\sigma$ by the Piola transformation
\begin{align}
	P(x) := (\det \nabla y(x)) \sigma(y(x)) (\nabla y(x))^{-T}.
\end{align}

To take into account the different response of different kind of materials to external forces
we will need constitutive equations to characterize the material behavior, e.g. if the material is stiff or flexible, elastic or plastic. In the present paper, we mainly consider elastic materials. A material is called elastic if it returns back to its undeformed state after removing the external forces. Mathematically, this behavior can be described as follows.

\begin{defi}[elastic material, {\cite[p. 89]{Cialet93}}]
A material response is called elastic if a mapping $\widehat{P} : \overline{\Omega} \times \R^{3 \times 3}_{>} \rightarrow \R^{3 \times 3}$, where $\R^{3 \times 3}_{>}$ is the set of real-valued $3\times3$-matrices with positive determinant, exists such that
\begin{align}
	P(x) = \widehat{P}(x, \nabla y(x))
\end{align}
for all $x \in \overline{\Omega}$. $\widehat{P}$ is then called the response function.
\end{defi}

From now on, we make three restrictions on the possible choices of the response function $\widehat{P}$. 
The first one is the general concept of material frame-indifference.
It states that the stress tensors as physical objects have to be invariant with respect to rotations of the observer.
This can be summarized in the axiom of material frame-indifference (cf. \cite[axiom 3.3-1]{Cialet93}): Let Q $\in$ SO(3) be arbitrarily chosen, then
\begin{align}
	t(Qy(x), Qn) = Qt(y(x),n)	
\end{align}
holds for all $x \in \overline{\Omega}$ and $n \in \mathbb{S}$.
The property of material frame-indifference can also be characterized with the help of the following theorem.
\begin{theorem}[material frame-indifference, {\cite[theorem 3.3-1]{Cialet93}}]
	Let $\widehat{P}$ be the response function to the first Piola-Kirchhoff stress tensor describing an elastic material behavior.
	Then the material satisfies the axiom of material frame-indifference if and only if
	\begin{align}
		\widehat{P}(x,QF) = Q \widehat{P}(x,F)	
	\end{align}
	holds for all $x \in \overline{\Omega}$, $F \in \R^{3 \times 3}_{>}$ and $Q \in SO(3)$.
\end{theorem}

The second restriction which we make is to consider only isotropic materials.
This means that the material's reaction is not depending on the direction of the external forces.
For example, a composite material with randomly distributed circular micro-structures has an isotropic behavior.
Analogously to the material frame-indifference, isotropic material behavior can also be characterized by considering the response function of the material.

\begin{defi}[isotropic material behavior, {\cite[p.106]{Cialet93}}]
	Let $\widehat{P}$ be the response function of the first Piola-Kirchhoff stress tensor.
	Then, the material behavior is called isotropic if
	\begin{align}
		\widehat{P}(x,FQ) = \widehat{P}(x,F)Q	
	\end{align}
	holds for all $x \in \overline{\Omega}$, $F \in \R^{3 \times 3}$ and $Q \in SO(3)$.
\end{defi}

The last restriction we make is to consider only hyperelastic materials.
\begin{defi}[hyperelastic material behavior, {\cite[p. 137]{Cialet93}}]
	A material response is called hyperelastic if there exists a stored energy function $\widehat{W}: \overline{\Omega} \times \R^{3 \times 3}_{>} \rightarrow \R$ such that
	\begin{align}
		\widehat{P}(x,F) = \dfrac{\partial \widehat{W}}{\partial F}(x,F)
		\label{eq:stored_energy}
	\end{align}
	is valid for all $x \in \overline{\Omega}$ and $F \in \R^{3 \times 3}_{>}$.
\end{defi}

The assumption of a hyperelastic material is often made in material modeling because it ensures the thermodynamic consistency of the material law if no dissipation is considered, i.e. for elastic materials (cf. \cite[chapter 13]{haupt02}).

For hyperelastic materials, the properties of material frame-indifference and isotropy can be characterized by the behavior of the stored energy function.
A stored energy function is called material frame-indifferent and isotropic, respectively, if the corresponding response function defined by \eqref{eq:stored_energy} is material frame-indifferent and isotropic, respectively.
For practical applications these properties are directly linked to the behavior of the stored energy function by the following theorem.

\begin{theorem}[{\cite[theorems 4.2-1 \& 4.3-1]{Cialet93}}]
	Let $\widehat{W}$ be a stored energy function of a hyperelastic material.
	Then, $\widehat{W}$ is material frame-indifferent if and only if
	\begin{align}
		\widehat{W}(x, QF) = \widehat{W}(x,F)	
	\end{align}
	holds for all $x \in \overline{\Omega}, F \in \R^{3 \times 3}$ and $Q \in SO(3)$. Furthermore, $\widehat{W}$ is isotropic if and only if
	\begin{align}
		\widehat{W}(x,FQ) = \widehat{W}(x,F)	
	\end{align}
	holds for all $x \in \overline{\Omega}, F \in \R^{3 \times 3}$ and $Q \in SO(3)$.
\end{theorem}

Finally, we assume that the reference configuration of the material is a natural state, i.e. we have $P_R(x) := \widehat{P}(x,I) = 0$, where $I$ is the unit matrix.
Furthermore, let $E \in \R^{3 \times 3}$ be the Green-St.Venant strain tensor defined by
\begin{align}
	E = E(u) := \dfrac{1}{2}((\nabla y)^T \nabla y - I) = \dfrac{1}{2}(\nabla u + (\nabla u)^T + (\nabla u)^T \nabla u)	
\end{align}
and $E^\text{lin}$ its linearization given by
\begin{align}
	E^\text{lin} = \nabla^s u.
\end{align}
Then the following representation for the stored energy function can be proven, which will be the basis for the derivation of the basic equations for linear hyperelastic isotropic materials.

\begin{theorem}[{\cite[theorem 4.5-1]{Cialet93}}]
	Consider a material with hyperelastic and isotropic response and let its reference configuration be a natural state.
	Then, there exist real-valued functions $\lambda$ and $\mu$ such that the stored energy function $\widehat{W}$ of this material can be written as
	\begin{align}
		\widehat{W}(x, \nabla y) = \overline{W}(x,E) = \dfrac{\lambda(x)}{2}(\trace(E))^2 + \mu(x) \trace(E^2) + o(\norm{E}^2)
		\label{eq:representation_energy_function}
	\end{align}
for $\norm{E}^2 \to 0$, where $	\trace(M)$ denotes the trace of the matrix $M$.
\end{theorem}

The function $\lambda$ is called first Lamé parameter and the function $\mu$ is called shear mo\-du\-lus or second Lamé parameter. The values of $\lambda$ and $\mu$ depend on the given material and are determined
experimentally. 

If only small deformations are considered, it is reasonable to neglect the small-$o$ part of \eqref{eq:stored_energy} and to use the linearization of the Green-St.Venant strain tensor instead of the full nonlinear tensor. Then, the stored energy function $\widehat{W}$ can be 
approximated by
\begin{align}
	\check{W}(x,\nabla^s u) = \dfrac{\lambda(x)}{2} \trace(\nabla^s u)^2 + \mu(x) \trace((\nabla^s u)^2)
	\label{eq:stored_energy_function_approx}	
\end{align}
and the first Piola-Kirchhoff stress tensor $P$ can be approximated by
\begin{align}
	\frac{\partial\check{W}}{\partial F}(x,\nabla^s u(x)) = \lambda(x) \trace(\nabla^s u(x)) I + 2\mu(x) \nabla^s u(x).
	\label{eq:piola_approx}	
\end{align}
For so-called linear hyperelastic isotropic materials the first Piola-Kirchhoff stress tensor $P$ exactly satisfies
\begin{align}
	P(x)  = \frac{\partial\check{W}}{\partial F}(x,\nabla^s u(x)) = \lambda(x) \trace(\nabla^s u(x)) I + 2\mu(x) \nabla^s u(x),
	\label{eq:piola_lin}	
\end{align}
which is Hooke's law for linear hyperelastic isotropic materials.
\medskip

By minimizing the energy $\check{W}$ we will derive the basic equations for linear hyperelastic isotropic materials for the experiment we have in mind, namely a tensile test, see figure \ref{fig:tt}.

\begin{figure}[htbp]
	\centering
	\includegraphics[width = 1.5in]{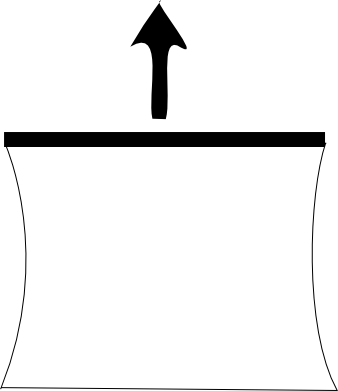}
	\caption{A tensile test}
	\label{fig:tt}
\end{figure}

Consider a fixed Cartesian coordinate system with an orthonormal basis $e_j$, $j = 1,2,3$. 
Let 
$$\Omega = (0,1)^2 \times (0,h)$$ 
for $h > 0$ and its boundary be decomposed by
$$\partial \Omega = \Gamma_1 \cup \Gamma_2 \;\, \text{ with} \;\;
\Gamma_1 := \{x \in \partial\Omega : x_2 := x \cdot e_2 \in \{0,1\}\}\;\, \text{ and} \;\;
\Gamma_2 := \partial\Omega \setminus \Gamma_1.$$

There are basically two different types of boundary conditions to model a tensile test in a physical reasonable way.
The first possibility is to fix the body on $\Gamma_1$, i.e. to chuck the body at the bottom and top end, while the rest of the body is free to move. 

The second possibility are so-called greased boundary conditions. 
Here, the body is fixed on $\Gamma_1$ only against moving in $e_2$ direction, i.e. the deformation of the body in $e_2$ direction is the only prescribed deformation and thus the body can move freely in $e_1$ and $e_3$ direction. Again on the remaining surface, the body is free to move. Hence, for a given displacement length $l > 0$, an admissible displacement $u$ should satisfy the Dirichlet boundary condition
\begin{align}
	u_2\vert_{\Gamma_1} = u_0[l] := \begin{cases}
 			l, & x_2 = 1, \\
 			0, & x_2 = 0.
 		\end{cases}
 	\label{eq:dirichlet_boundary}
\end{align}
Because even in the simplest case of a homogeneous material, the deformation beha\-vi\-or of a body with the first type of boundary conditions is quite complex and thus hard to analyze, the greased boundary conditions are often preferred.

Moreover, we introduce the so-called strain energy
\begin{align*}
	u \mapsto J[u] := \int_\Omega \check{W}(x,\nabla u) dx\,,	
\end{align*}
where $\check{W}$ is given by \eqref{eq:stored_energy_function_approx} and $u$ should satisfy the boundary condition \eqref{eq:dirichlet_boundary}. The basic equations for linear hyperelastic isotropic materials in a tensile test are the
Euler-Lagrange equations with respect to the functional $J$. 
For a minimizer $u$ of $J$ satisfying \eqref{eq:dirichlet_boundary} it holds
\begin{align*}
	J[u + \varepsilon \phi] \geq J[u] 
\end{align*}
for all $\phi \in X	$, where 
\begin{align*}
	X := C^\infty(\overline{\Omega}, \R) \times C^\infty_{\Gamma_1}(\overline{\Omega},\R) \times C^\infty(\overline{\Omega},\R)	
\end{align*}
with 
\begin{align*}
	C^\infty_{\Gamma_1}(\overline{\Omega},\R) := \{f \in C^\infty(\overline{\Omega},\R) : f \text{ vanishes in a neighborhood of } \Gamma_1\}
\end{align*}
and $\varepsilon \in (-\varepsilon_0, \varepsilon_0)$ with $\varepsilon_0 > 0$. 
Notice that $X$ is chosen such that $u + \varepsilon \phi$ also satisfies the boundary condition \eqref{eq:dirichlet_boundary}.
A necessary condition for such a minimum is 
\begin{align*}
	\left.\dfrac{d}{d\varepsilon} J[u + \varepsilon \phi]\right\vert_{\varepsilon = 0} = 0.
\end{align*}
Explicit calculations using Green's formula yield
\begin{align}
	- \div\left(\dfrac{\partial\check{W}}{\partial F}(x,\nabla u)\right) = 0	\qquad \text{ in } \Omega
	\label{eq:3D_equation}
\end{align}
with the Neumann boundary conditions
\begin{align}
	\left(\dfrac{\partial\check{W}}{\partial F}(x,\nabla u)\right)n &= 0 \qquad \text{ on } \Gamma_2, \label{eq:3D_neumann_1} \\
	\left[\left(\dfrac{\partial\check{W}}{\partial F}(x,\nabla u)\right)n\right]_i &= 0 \qquad \text{ on } \Gamma_1 \text{ for } i = 1,3. \label{eq:3D_neumann_2}
\end{align}
Then, by using \eqref{eq:stored_energy_function_approx}, we obtain 
\begin{align}
	-\div(\lambda(x) \trace(\nabla^s u(x)) I + 2 \mu(x) \nabla^s u(x)) &= 0 \qquad\qquad \text{ in } \Omega, \label{eq:DGL3D}\\
	(\lambda(x) \trace(\nabla^s u(x)) I + 2 \mu(x) \nabla^s u(x))n &= 0 \qquad\qquad \text{ on } \Gamma_2, \label{eq:3DBC1} \\
	[(\lambda(x) \trace(\nabla^s u(x)) I + 2 \mu(x) \nabla^s u(x))n]_i &= 0 \qquad\qquad \text{ on } \Gamma_1 \text{ for } i = 1,3, \label{eq:3DBC2} \\
	u_2 &= u_0[l] \qquad\,\, \text{ on } \Gamma_1, \label{eq:3DBC3}
\end{align}
which are the basic equations for linear hyperelastic isotropic solids applied to the experimental setup of a three dimensional tensile test.

Next, the above derived three dimensional model is reduced to a two dimensional one by considering the so-called plane stress state (cf. \cite[chapter 2]{Galin08}).
The plane-stress assumption, i.e. all occurring stress vectors lie in one plane, is reasonable for $h$ very small. This corresponds to a thin plate, where the thickness is very small compared to the other dimensions.
The main assumption of the plane-stress state is that the first Piola-Kirchhoff stress tensor has the form
\begin{align*}
	P = \left(\begin{array}{ccc}
		P_{11} & P_{12} & 0 \\
		P_{21} & P_{22} & 0 \\
		0 & 0 & 0
	\end{array}\right).
\end{align*}
Then \eqref{eq:piola_lin} implies 
\begin{align*}
	0 &= P_{33} = \lambda (E^{\text{lin}}_{11} + E^{\text{lin}}_{22} + E^{\text{lin}}_{33}) + 2 \mu E^{\text{lin}}_{33}
\end{align*}	
and therefore
\begin{align*}	
	 E^{\text{lin}}_{33} &= - \dfrac{\lambda}{\lambda + 2\mu} (E^{\text{lin}}_{11} + E^{\text{lin}}_{22}).
\end{align*}
This yields 
\begin{align*}
	P_{ii} &= \lambda (E^{\text{lin}}_{11} + E^{\text{lin}}_{22}) - \lambda \dfrac{\lambda}{\lambda + 2\mu} (E^{\text{lin}}_{11} + E^{\text{lin}}_{22}) + 2 \mu E^{\text{lin}}_{ii} \\
	&= \underbrace{\lambda \left(1 - \dfrac{\lambda}{\lambda + 2\mu}\right)}_{=: \lambda_\text{eff}} (E^{\text{lin}}_{11} + E^{\text{lin}}_{22}) + 2 \mu E^{\text{lin}}_{ii}
\end{align*}
for $i = 1,2$.
Hence, by introducing 
\begin{align*}
	\widetilde{P} := \left(\begin{array}{cc}
		P_{11} & P_{12} \\
		P_{21} & P_{22}
	\end{array}\right) \text{ and }
	\widetilde{E}^{\text{lin}} := \left(\begin{array}{cc}
		E^{\text{lin}}_{11} & E^{\text{lin}}_{12} \\
		E^{\text{lin}}_{21} & E^{\text{lin}}_{22}
	\end{array}\right)
\end{align*}
we obtain
\begin{align}
	\widetilde{P} = \lambda_\text{eff} \trace(\widetilde{E}^{\text{lin}}) I + 2\mu \widetilde{E}^{\text{lin}}.
	\label{eq:reduced_stress_equation}
\end{align}

Therefore, the plane-stress formulation of a two dimensional tensile test with greased boundary conditions for linear hyperelastic isotropic materials whose reference configuration is a natural state is given by
\begin{align}
	-\div(\lambda_\text{eff}(x) \trace(\nabla^s u(x)) I + 2\mu(x) \nabla^s u(x)) &= 0 \qquad\qquad  \text{ in } \widetilde{\Omega}, \label{eq:DGL2D}\\
	(\lambda_\text{eff}(x) \trace(\nabla^s u(x)) I + 2\mu(x) \nabla^s u(x))n &= 0 \qquad\qquad \text{ on } \widetilde{\Gamma}_2, \label{eq:BC2D1}\\
	\left[(\lambda_\text{eff}(x) \trace(\nabla^s u(x)) I + 2\mu(x) \nabla^s u(x))n\right]_1 &= 0 \qquad\qquad \text{ on } \widetilde{\Gamma}_1, \label{eq:BC2D2} \\
	u_2 &= u_0[l] \qquad\,\, \text{ on } \widetilde{\Gamma}_1, \label{eq:BC2D3}
\end{align}
where $\widetilde{\Omega}= \{x \in \Omega|x_3=0\}, \widetilde{\Gamma}_i= \{x \in \Gamma_i|x_3=0\}, i=1,2$, $u: \overline{\widetilde{\Omega}} \to \R^{2}$ and $u_0[l]$ is as in \eqref{eq:dirichlet_boundary}.
For notational simplicity, we write $P, E^{\text{lin}},\Omega,\Gamma_i$ and $\lambda$ instead of $\widetilde{P}, \widetilde{E}^{\text{lin}}, \widetilde{\Omega}, \widetilde{\Gamma}_i$ and $\lambda_\text{eff}$
from now on.

Finally, we present the one dimensional reduction of the above model equations. In this case, 
we obtain
\begin{align}
	- \dfrac{d}{dx} \Big((\lambda(x) + 2 \mu(x)) \dfrac{d}{dx} u(x)\Big) &= 0 \qquad\qquad  \text{ in } \Omega:=(0,1), \label{eq:DGL1D}
\\
u &= u_0[l] \qquad\,\, \text{ on } \partial\Omega, \label{BC1D}	
\end{align}
where $u_0[l]$ is as in \eqref{eq:dirichlet_boundary}.
By introducing 
$$\kappa= \lambda+2 \mu,$$ 
the so-called longitudinal modulus, we arrive at
\begin{align}
	- \dfrac{d}{dx} \Big(\kappa(x) \dfrac{d}{dx} u(x)\Big) &= 0 \qquad\qquad  \text{ in } \Omega, 
	\label{eq:dgl_in_terms_of_u} \\
u &= u_0[l] \qquad\,\, \text{ on } \partial\Omega. \label{eq:BC1D}	
\end{align}

\subsection{Existence and uniqueness of solutions}
After deriving the basic equations for linear hyperelastic isotropic solids applied to a tensile test, 
we now present an existence and uniqueness result. Let $d \in \{1,2,3\}$ and $\Omega= (0,1)$
for $d=1$, $\Omega= (0,1)^{2}$ for $d=2$ and $\Omega= (0,1)^{2} \times (0,h),\, h>0$,
for $d=3$.
First, we introduce the function spaces needed to formulate the existence and uniqueness result. These are the Sobolev spaces $H^k(\Omega)$ equipped with the Sobolev norm
\begin{align*}
	\norm[H^k(\Omega)]{u} := \left(\sum_{\snorm{\alpha} \leq k} \int_\Omega \snorm{\partial^\alpha u}^2 dx\right)^\frac{1}{2}
\end{align*}
and the space
\begin{align*}
	H^1_m(\Omega) &:= \left\{u \in H^1(\Omega) : \dfrac{1}{\snorm{\Omega}} \int_\Omega u dx = 0\right\}, 
\end{align*}
which is a closed subspace $H^1(\Omega)$.
Moreover, with the help of these spaces, the following function spaces can be introduced.

\begin{defi}[admissible function spaces]\label{def:admissible_spaces}
	Let $l \in \R$ be given. Then the space $\calW_l$ is defined by
\begin{align*}
		\calW_l := \left\{u \in H^1(\Omega): u \text{ satisfies \eqref{eq:dirichlet_boundary}} \text{ almost everywhere in the trace sense }\right\}
	\end{align*}	
if $d=1$, and	
	\begin{align*}
		\calW_l := \left\{u \in H^1_m(\Omega) \times (H^1(\Omega))^{d-1} : u_2 \text{ satisfies \eqref{eq:dirichlet_boundary}} \text{ almost everywhere in the trace sense }\right\}
	\end{align*}
if $d >1$,	
	equipped with the $H^{1}$-norm. Moreover,
	 its dual space is denoted by $\calW_l'$.
\end{defi}

By using this definition, the notion of a weak solution to the system \eqref{eq:DGL3D} - \eqref{eq:3DBC3} and its one and two dimensional reductions can be formulated. This notion is motivated by the variational argument used for deriving the system of equations.

\begin{defi}[weak solution]\label{def:weak_solution}
	Let $l \in \R$ be given. Then $u \in \calW_l$ is called a weak solution of \eqref{eq:DGL3D} - \eqref{eq:3DBC3} for $d=3$, \eqref{eq:DGL2D} - \eqref{eq:BC2D3}
for $d=2$ and 	\eqref{eq:DGL1D} - \eqref{BC1D} for $d=1$, respectively if
	\begin{align}
		\int_\Omega \lambda \trace(\nabla^s u) \trace(\nabla^s v) + 2\mu \nabla^s u : \nabla^s v dx = 0,
		\label{eq:weak_solution}
	\end{align}
where $A:B= \trace (A^{T}B)$, for all $v \in \calW_0$.
\end{defi}

\textbf{Remark:} The condition that $u_1$ has to have zero mean is introduced as an artificial condition to ensure uniqueness of the solution.
However, this extra condition only excludes additional superimposed rigid body motions in $x_1$-direction.
Therefore, the solution without the mean value condition is still unique up to a superimposed rigid body motion in $x_1$-direction which is physically completely reasonable since no Dirichlet conditions are given for $u_1$ and hence, complete uniqueness of the solution cannot be expected.
\\[-2mm]

Now, the existence and uniqueness theorem can be stated.
The theorem is formulated in a very general way, which is needed to ensure the well-posedness of the equations used in the perturbation theory in section \ref{chap:perturbation_theory}.

\begin{theorem}[existence and uniqueness]\label{thm:existence_uniqueness}
	Let $l \in \R$ be a given prescribed displacement, $\lambda, \mu \in L^\infty(\Omega)$ with
	\begin{align}
		\lambda &\geq 0 \,  \text { almost everywhere in }\, \Omega, \label{lbound}\\
		\mu &\geq \mu^{\ast} > 0 \, \text { almost everywhere in }\, \Omega, \label{mbound}
\end{align}		
	and $F \in \calW_0'$ satisfying
	\begin{align}
		F(u + c e_1) = F(u) \label{eq:invarience_right_side}
	\end{align}
	for all $c \in \R$. Then there exists a unique $u \in \calW_l$ such that
	\begin{align}
	\label{weakeq}
		\int_\Omega \lambda \trace(\nabla^s u) \trace(\nabla^s v) + 2\mu \nabla^s u : \nabla^s v dx = F(v)
	\end{align}	
	holds for all $v \in \calW_0$. Moreover, $u$ satisfies
	\begin{align}
	\label{ellest}
	\|u\|_{H^{1}(\Omega)} \leq C\, (\, \|F\|_{\calW_0'} + (\|\lambda\|_{L^{\infty}(\Omega)}+ 2 \|\mu\|_{L^{\infty}(\Omega)})\,|l|\, )
	\end{align}
	for a constant $C>0$ which depends on $\mu^{\ast}$ but is independent of $F,l,\lambda,\mu$ and $u$.
\end{theorem}

The invariance condition \eqref{eq:invarience_right_side} is necessary to reduce the set of test functions to $H^1$-functions with zero mean.
The existence of a unique weak solution to \eqref{eq:DGL3D} - \eqref{eq:3DBC3} for $d=3$, \eqref{eq:DGL2D} - \eqref{eq:BC2D3}
for $d=2$ and 	\eqref{eq:DGL1D} - \eqref{BC1D} for $d=1$, respectively, in the sense of definition \ref{def:weak_solution} is a direct consequence of the above theorem by choosing $F \equiv 0$. 

The proof of Theorem \ref{thm:existence_uniqueness} relies on standard arguments from elliptic theory. For zero Dirichlet boundary conditions, i.e. for $l=0$, the unique existence of a solution $u$ and the validity of estimate \eqref{ellest} follows from Lax-Milgram's theorem. In two and three dimensions, the coerciveness of the bilinear mapping defined by the left-hand side of \eqref{weakeq}
is verified with the help of Korn's inequalities. For $l \neq 0$, the assertions of the theorem
follow from the case with $l=0$ by decomposing $u$ into $u=z+\widetilde{u}$ with $\widetilde{u} \in  \calW_l$. For further details, we refer, for example, to \cite{Cialet93}, \cite{CD99} and the references therein.

\subsection{The tensile force}\label{sec:tensile_force}

An important mechanical property of materials which are measured by a tensile test are their stress-strain curves. For a given force, the displacement $l>0$ at the top end of the material 
has to be measured. Since the displacement is uniquely determined by the force, it is equivalent to prescribe the displacement $l>0$ at the top end and determine the tensile force, i.e. the force needed to deform the material. Hence, for an analytical or a numerical investigation of a tensile test we have to derive a formula for the tensile force from our model equations. 

For sufficiently regular solutions $u$ of our model equations the tensile force $F$ is given by
\begin{align}
	F = F[\lambda, \mu, l] &= \int_{\Gamma_1^\text{top}} (Pn)_2 do,
\end{align} 
where ${\Gamma_1^\text{top}} := \{x \in \Gamma_1 : x_2 = 1\}$.
In the two dimensional case, this formula is equivalent to 
\begin{align}
	F[\lambda, \mu, l]  
   &= \int_{\Gamma_1^\text{top}} \lambda \partial_1 u_1 + (\lambda + 2\mu) \partial_2 u_2 do
\end{align} 
and in the one dimensional case, this formula reduces to 
\begin{align}
	F = F[\kappa,l] = \kappa(1) u'(1)\,.
	\label{eq:force-1d}
\end{align}

However, in the case of a composite material, $\lambda$ and $\mu$ are step functions and the solutions $u$ exist only in the weak sense. Therefore, in two and three dimensions, the above formula  might not be well-posed since we have not enough regularity to define the trace of $\nabla^s u$ on $\partial\Omega$. Nevertheless, in this situation a different expression for the force can be derived, which we will show in the following for the two dimensional case. 
In order to do so, we consider the function space
\begin{align*}
	\mathcal{E}(\Omega) = \left\{v \in L^2(\Omega) : \div v \in L^2(\Omega)\right\}
\end{align*}
Then, the following result holds.

\begin{lem}\label{cor:PinE}
	Let $u \in \calW_l$ be the weak solution of \eqref{eq:DGL2D} - \eqref{eq:BC2D3} and $P$ the corresponding first Piola-Kirchhoff stress tensor. Then we have $P \in \mathcal{E}(\Omega)$.
\end{lem}
\begin{proof}
	Because  $u \in \calW_l$ is the weak solution of \eqref{eq:DGL2D} - \eqref{eq:BC2D3}, we have
	\begin{align*}
		\int_\Omega \nabla^s v : P dx = 0
	\end{align*}
	for all $v \in \calW_0$.
	Since $C^\infty_0(\Omega) \subset \calW_0$, this identity also holds for all $v \in C^\infty_0(\Omega)$.
	Hence, the above statement yields that $\div(P) = 0$ in the distributional sense, which implies the assertion of the lemma.
\end{proof}

In this case, according to \cite[theorem 1.2]{temam01}\footnote{In this theorem, it is assumed that the domain to have $C^2$-boundary. However, \cite[remark 1.3]{temam01} yields the extension of this result to domains with Lipschitz boundary.}, there exists a linear, bounded trace operator
\begin{align*}
	T_\mathcal{E} : \mathcal{E}(\Omega) \rightarrow H^{-1/2}(\partial\Omega)
\end{align*}
with
\begin{align*}
	T_\mathcal{E}[v] = vn\vert_{\partial\Omega}
\end{align*}
for all $v \in C^\infty_0(\overline{\Omega})$.
Here, $H^{-1/2}(\Omega)$ is the dual space of $H^{1/2}(\Omega) := \calR(T_{H^1})$, where $T_{H^1}$ is the trace operator on $H^1(\Omega)$ and $\calR$ denotes the range of an operator. 
Furthermore, it is proven in \cite[theorem 1.2]{temam01} that a generalized Stokes formula holds, which is given by
\begin{align}
	\scalarprod{v}{\nabla w}_{L^2(\Omega)} + \scalarprod{\div v}{w}_{L^2(\Omega)} = \scalarprod{T_\mathcal{E}[v]}{T_{H^1}[w]}_{H^{-1/2}(\Omega)}
	\label{eq:generalized_stokes}
\end{align}
for all $v \in E(\Omega)$ and $w \in H^1(\Omega)$. Here, the mapping
\begin{align*}
	\scalarprod{\cdot}{\cdot}_{H^{-1/2}(\Omega)} : H^{-1/2}(\Omega) \times H^{1/2}(\Omega) \rightarrow \R
\end{align*}
is called the duality mapping.
Using this trace operator and the generalized Stokes formula, the following theorem can be proven.

\begin{theorem}[tensile force]\label{thm:tensile_force}
	Let $l \in \R$, $\lambda, \mu \in L^\infty(\Omega)$ satisfy the assumptions from theorem \ref{thm:existence_uniqueness}.
	Then the tensile force at the top end of the material is given by
	\begin{align}
		F[\lambda,\mu,l] = \int_\Omega P_{22} dx,
		\label{eq:tensile_force}
	\end{align}
	where $P$ is the first Piola-Kirchhoff stress tensor.
\end{theorem}
\begin{proof}
	Define $w = (0,x_2)^T$, which is a $H^1$-function.
	Furthermore, with lemma \ref{cor:PinE} and \eqref{eq:BC2D1} we get
	\begin{align*}
		F[\lambda, \mu, l] &= \int_{\Gamma_1^\text{top}} (Pn)_2 do = \int_{\partial\Omega} (Pn) \cdot w do = \scalarprod{T_\mathcal{E}[P]}{T_{H^1}[w]}_{H^{-1/2}(\Omega)} \\
		&\stackrel{\eqref{eq:generalized_stokes}}{=} \int_{\Omega} \underbrace{\div(P)}_{=0} \cdot w dx + \int_\Omega \nabla w : P dx = \int_\Omega \nabla w : P dx = \int_\Omega P_{22} dx,
	\end{align*}
	which is the assertion.
\end{proof}

\section{The embedded cell method}\label{chap:ecm}
The embedded cell method is a numerical scheme to compute the stress-strain curves of composite materials by performing a numerical homogenization process. In this paper, we will consider composite materials consisting, for instance, of a metal matrix material with embedded ceramic particles. The general idea of homogenization is to derive effective properties of a multiscale problem by replacing the problem by an appropriate homogeneous problem.
The multiscale properties of metal-ceramic composite materials are a consequence of the different values of the material parameters $\lambda$ and $\mu$ of the embedded material and the surrounding material.
As discussed for example in \cite{CD99}, the mathematical analysis of effective properties of multiscale materials is performed by considering sequences of faster and faster oscillating material parameters.

Let, for example, $\lambda, \mu : \R^d \rightarrow \R$, where $d \in \{1,2,3\}$, be periodic functions and define for $\delta > 0$:
\begin{align*}
	\lambda^\delta(x) := \lambda\left(\dfrac{x}{\delta}\right), \qquad\quad
	\mu^\delta(x) := \mu\left(\dfrac{x}{\delta}\right).
\end{align*}
Then, the homogenization problem is given by the family of solutions $(u^\delta)_{\delta > 0} \subset \calW_l$ to 
\begin{align}
	\int_\Omega \lambda^\delta \trace(\nabla^s v) \trace(\nabla^s u^\delta) + 2 \mu^\delta \nabla^s v : \nabla^s u^\delta dx = 0
\label{anahom}
\end{align}
for all $v \in \calW_0$.
Notice that this family of solutions is well-posed according to theorem \ref{thm:existence_uniqueness}.

The aim of analytical homogenization theory is to examine the asymptotic behavior of the solutions $(u^\delta)_{\delta > 0}$ for $\delta \rightarrow 0$.
In particular, it is of great interest to study if a limit exists, which equation the possible limit solves and which are the homogeneous coefficients of this effective equation, the so-called effective material parameters. The idea behind this is that the effective behavior of a multiscale material should be given by such a limit and by the effective material parameters of its equation.
An essential result of analytical homogenization theory is the following theorem.

\begin{theorem}[homogenization]\label{thm:homogenization_convergence_result}
	Let $d \in \{1,2,3\}$, $l \in \R$, $(\lambda^{\delta})_{\delta > 0}, (\mu^{\delta})_{\delta > 0} \subset L^\infty(\Omega)$ be fa\-milies of material parameters defined as above and $(u^{\delta})_{\delta > 0} \subset \calW_l$ be the correspon\-ding fa\-mily of solutions to \eqref{anahom}.
	Then there exists a unique $u^{\text{hom}} \in \calW_l$ and constants $\lambda^{\text{hom}}, \mu^{\text{hom}}$ $\in \R$, called effective material parameters, such that 
	\begin{align} \label{3eff1}
		\begin{cases}
			u^{\delta} \rightharpoonup u^{\text{hom}} & \text{ in } \calW_l \\
			\lambda^{\delta} \trace(\nabla^s u^{\delta}) I + 2 \mu^{\delta} \nabla^s u^{\delta} \rightharpoonup \lambda^{\text{hom}} \trace(\nabla^s u^{\text{hom}}) I + 2\mu^
			{\text{hom}} \nabla^s u^{\text{hom}} & \text{ in } L^2(\Omega, \R^{d\times d})
		\end{cases}
	\end{align}
for $\delta \rightarrow 0$. Moreover, we have
	\begin{align} \label{3eff2}
		\int_\Omega \lambda^{\text{hom}} \trace(\nabla^s u^{\text{hom}}) \trace(\nabla^s v) + 2\mu^{\text{hom}} \nabla^s u^{\text{hom}} : \nabla^s v dx = 0
	\end{align}
	for all $v \in \calW_0$ and consequently
\begin{align} \label{3eff3}
\lim_{\delta \to 0} F[\lambda^{\delta},\mu^{\delta},l] &= F[\lambda^{\text{hom}},\mu^{\text{hom}},l] =: F^{\text{hom}}.
\end{align}	
\end{theorem}

\begin{proof}
	The result is proven in \cite[theorem 10.11]{CD99} for different boundary conditions.
	However, the proof in the case of our boundary conditions can be done analogously.
\end{proof}

However, in general, it is difficult to compute $\lambda^{\text{hom}}, \mu^{\text{hom}}$ and  $u^{\text{hom}}$, and to appro\-ximate $u^{\text{hom}}$ by $u^{\delta}$ with small $\delta$ would be hard to handle numerically because of the very fast oscillating coefficients. Therefore, instead of solving the above analytical homogenization problem numerically, the basic idea of the embedded cell method is to replace the complex geometry of this problem by a so-called embedded cell, which is embedded into a dummy material, see figure \ref{fig:emc}.

\begin{figure}
	\centering
	\includegraphics[width = 0.58\textwidth]{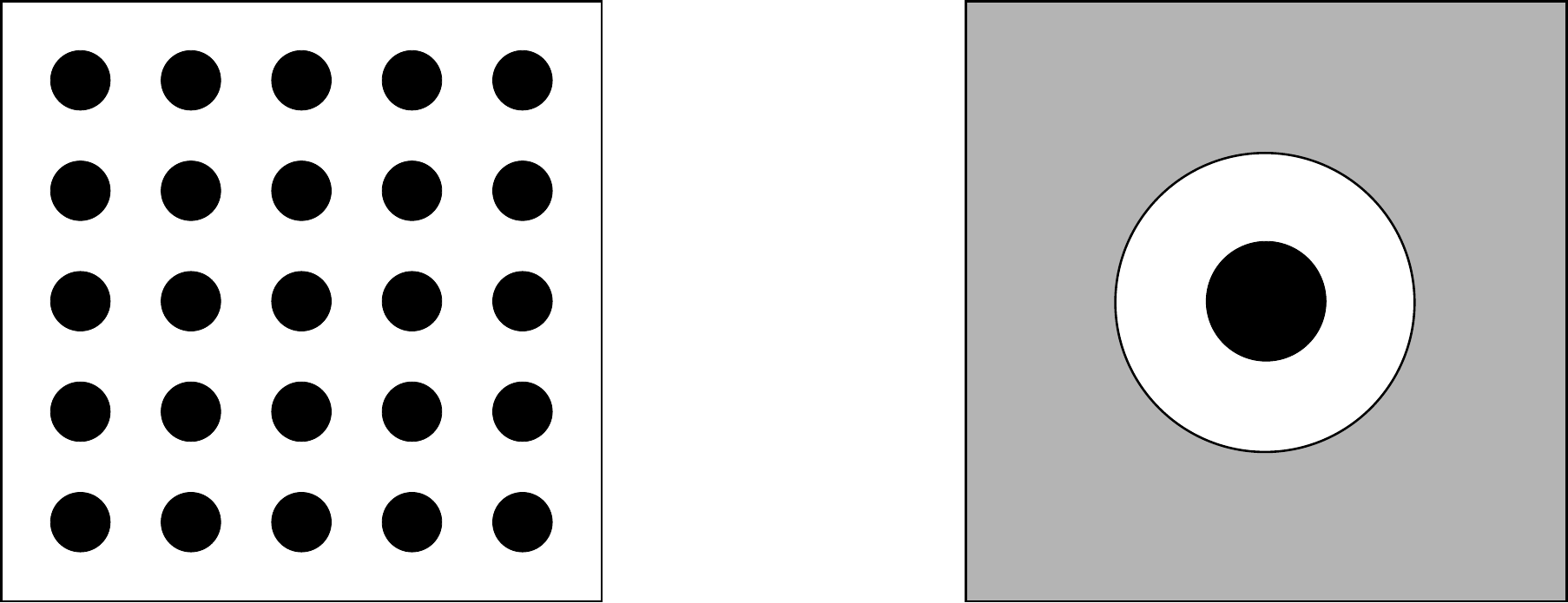}
	\caption[Full material and embedded cell]{Left: Composite with periodically distributed ceramic particles (black) embedded into a metal matrix (white). Right: Embedded cell in a dummy material (grey)}
	\label{fig:emc}
\end{figure}

The material parameter of this dummy material is iteratively determined by the following scheme. In the one dimensional case, this scheme is given by algorithm \ref{alg:emc_1d}.
\vspace{3mm}

\begin{algorithm}[H]\label{alg:emc_1d}
	\KwData{material parameters $\kappam, \kappac$, deformation length $l$}
	\KwResult{dummy material parameter $\kdummy$}
	Make initial guess for $\kdummy(0)$\;
	\For{$n = 1,2,\dots$}{
		Find equivalent material parameter $\kappa^{\text{equiv}}(n)$ for the material with dummy material parameter $\kdummy(n-1)$\;
		Set $\kdummy(n) = \kappa^{\text{equiv}}(n)$\;
	}
\caption{embedded cell algorithm for one dimensional linear hyperelastic isotropic materials} 
\end{algorithm}	
\vspace{3mm}

It is clear that the crucial step in algorithm \ref{alg:emc_1d} is to determine the equivalent longitudinal modulus $\kappa^{\text{equiv}}$ which is defined the following way.

\begin{defi}[equivalent longitudinal modulus]\label{def:equivalent_material_param_1d}
	Consider a one dimensional linear hyperelastic isotropic material with longitudinal modulus $\kappa = \kappa(x)$ and let $F[\kappa,l]$ be the corresponding tensile force. 
	Then, $\kappa^\text{equiv} \in \R^+$ is called equivalent longitudinal modulus if
	\begin{align}
	F[\kappa^\text{equiv},l] = F[\kappa,l].
	\end{align}
\end{defi}

Since the deformation of a homogeneous one dimensional linear hyperelastic isotro\-pic material with deformation length $l>0$ is given by $u(x)=lx$, we obtain by using \eqref{eq:force-1d} that
\begin{align}
	\kappa^\text{equiv} = F[\kappa,l]l^{-1}.
	\end{align} 

To discuss the two dimensional case we restrict ourselves for simplicity to the situation of a composite material with a constant shear modulus. Then, the scheme is given by algorithm \ref{alg:emc}.
\vspace{0.3cm}
	
\begin{algorithm}[H]\label{alg:emc}	
	\KwData{material parameters $\lambdam, \lambdac$ and $\mu$, deformation length $l$}
	\KwResult{dummy material parameter $\ldummy$}
	Make initial guess for $\ldummy(0)$\;
	\For{$n = 1,2,\dots$}{
		Find equivalent material parameter $\lambda^\text{equiv}(n)$ for the material with dummy material parameter $\ldummy(n-1)$\;
		Set $\ldummy(n) = \lambda^\text{equiv}(n)$\;
	}
	\caption{embedded cell algorithm for two dimensional linear hyperelastic isotropic materials with constant shear modulus} 
\end{algorithm}
\vspace{0.3cm}

In algorithm \ref{alg:emc}, the crucial step is to determine the equivalent first Lamé parameter $\lambda^\text{equiv}$, which is defined the following way.

\begin{defi}[equivalent material parameter]\label{def:equivalent_material_param}
	Consider a linear hyperelastic isotropic material with constant shear modulus $\mu$ and first Lamé parameter $\lambda = \lambda(x)$ and let $F[\lambda,\mu,l]$ be the corresponding tensile force.
	Then, $\lambda^\text{equiv} \in \R^+$ is called equivalent first Lamé parameter if
	\begin{align*}
		F[\lambda^\text{equiv},\mu,l] = F[\lambda, \mu, l]. 
	\end{align*}
\end{defi}

The existence of such an equivalent first Lamé parameter for small perturbations, the formula
\begin{align}
	\lambda^\text{equiv} = \left(F[\lambda, \mu, l]- 2\mu l\right) \left(2l - \dfrac{F[\lambda, \mu, l]}{2\mu}\right)^{-1}
	\end{align} 
as well as the fact that $\lambda^\text{equiv}$ is independent of the choice of $l$ is proven in \ref{chap:appendix_well-posedness}.

Next, the geometry of the embedded cell should be specified in more detail.
Therefore, let $\Omega_m, \Omega_c \subset \Omega$ be the areas of the metal and the ceramic, respectively, and $\Omega_c \cup \Omega_m = \Omega$.
Then a decomposition of $\Omega$ is introduced by $\Omega = \widetilde{\Omega}_m \cup \widetilde{\Omega}_c \cup \widetilde{\Omega}_\text{dummy}$ where $\widetilde{\Omega}_m, \widetilde{\Omega}_c, \widetilde{\Omega}_\text{dummy}$ are pairwise disjoint subdomains of $\Omega$.
According to \cite{Schmauder96}, this decomposition is defined in the following way:
\begin{align*}
	\widetilde{\Omega}_c &:= \{x \in \Omega : x \in B_{r_1}(\bar{x})\}, \\
	\widetilde{\Omega}_m &:= \{x \in \Omega : x \in B_{r_2}(\bar{x}) \setminus B_{r_1}(\bar{x})\}, \\
	\widetilde{\Omega}_\text{dummy} &:= \Omega \setminus (\widetilde{\Omega}_m \cup \widetilde{\Omega}_c)
\end{align*}
with $\bar{x} = (0.5,0.5)^T$, $B_r(\bar{x}):= \{ x \in \R^{d}|\,|x-\bar{x}| \leq r\}$ and $0 < r_1 \leq r_2 \leq 0.1$ chosen such that
\begin{align}
	\dfrac{\snorm{\Omega_c}}{\snorm{\Omega_m}} = \dfrac{\snorm{\widetilde{\Omega}_c}}{\snorm{\widetilde{\Omega}_m}},
	\label{eq:volume_fraction_ecm}
\end{align}
where $|G|$ denotes the ($d$-dimensional) volume of $G \subset \R^{d}$.

With this specifications the embedded cell method can be formulated in all details. In the one dimensional case, it is given by algorithm \ref{alg:emc_detail_1d}.
\vspace{3mm}

\begin{algorithm}[H]\label{alg:emc_detail_1d}	
\KwData{material parameters $\kappam, \kappac$, deformation length $l$}
	\KwResult{dummy material parameter $\kdummy$, tensile force $F^\text{ECM}$}
	Make initial guess for $\kdummy(0)$ by defining
	\\[-5mm]
	\begin{align} \label{kd0}
		\kdummy(0) := \snorm{\Omega_c} \kappac + \snorm{\Omega_m} \kappam\text{\;}
	\end{align}
	\\[-3mm]
	\For{$n = 1,2,\dots$}{
		Compute the displacement $u^{n-1} \in \calW_l$ by solving
		\begin{align}
			\int_\Omega \kappa^{n-1} (u^{n-1})' v' dx = 0
			\label{eq:weak_form_ECM_1d}
		\end{align}
		for all $v \in \calW_0$ with
		\begin{align} \label{kappa_ECM}
			\kappa^{n-1} = \chi_{\widetilde{\Omega}_c} \kappac + \chi_{\widetilde{\Omega}_m} \kappam + \chi_{\widetilde{\Omega}_\text{dummy}} \kdummy(n-1)\text{\;}
		\end{align}
		Compute the tensile force $F^n:=F[\kappa^{n-1},l]$ from $u^{n-1}$\;
		Regain the equivalent material parameter from $F^n$ by using
\\[-6mm]		
		\begin{align} \label{kequn}
			\kappa^\text{equiv}(n) = F^n l^{-1}\text{\;}
		\end{align}
\\[-4mm]		
		Set 
\begin{align}		\label{kdn}
		\kdummy(n) := \kappa^\text{equiv}(n)\text{\;}
		\end{align}
\\[-4mm]	
	}
Define
\begin{align}
	\kdummy &:= \lim_{n \rightarrow \infty} \kdummy(n) \label{eq:lim_parameter_1d}\text{\;}\\
	F^\text{ECM} &:= \lim_{n \rightarrow \infty} F^n\text{\;} \label{eq:lim_force_1d}	
\end{align}
\caption{embedded cell algorithm for one dimensional linear hyperelastic isotropic materials (detailed)} 
\end{algorithm}
\vspace{3mm}

In the two dimensional case with constant shear modulus, the embedded cell method is given by algorithm \ref{alg:emc_detail}.
\vspace{3mm}

\begin{algorithm}[H]\label{alg:emc_detail}
\KwData{material parameters $\lambdam, \lambdac$ and $\mu$, deformation length $l$}
	\KwResult{dummy material parameter $\ldummy$, tensile force $F^\text{ECM}$}
	Make initial guess for $\ldummy(0)$ by defining
	\\[-5mm]
	\begin{align} \label{ld0}
		\ldummy(0) := \snorm{\Omega_c} \lambdac + \snorm{\Omega_m} \lambdam\text{\;}
	\end{align}
	\\[-3mm]
	\For{$n = 1,2,\dots$}{
		Compute the displacement $u^{n-1} \in \calW_l$ by solving
		\begin{align}
			\int_\Omega \lambda^{n-1} \trace(\nabla^s u^{n-1}) \trace(\nabla^s v) + 2\mu \nabla^s u^{n-1} : \nabla^s v dx = 0
			\label{eq:weak_form_ECM}
		\end{align}
		for all $v \in \calW_0$ with
		\begin{align} \label{lambda_ECM}
			\lambda^{n-1} = \chi_{\widetilde{\Omega}_c} \lambdac + \chi_{\widetilde{\Omega}_m} \lambdam + \chi_{\widetilde{\Omega}_\text{dummy}} \ldummy(n-1)\text{\;}
		\end{align}
		Compute the tensile force $F^n:=F[\lambda^{n-1},\mu,l]$ from $u^{n-1}$\;
		Regain the equivalent material parameter from $F^n$ by using
\\[-6mm]		
		\begin{align} \label{lequn}
			\lambda^\text{equiv}(n) = \left(F^n- 2\mu l\right) \left(2l - \dfrac{F^n}{2\mu}\right)^{-1}\text{\;}
		\end{align}
\\[-4mm]		
		Set 
\begin{align}		\label{ldn}
		\ldummy(n) := \lambda^\text{equiv}(n)\text{\;}
		\end{align}
\\[-4mm]	
	}
Define
\begin{align}
	\ldummy &:= \lim_{n \rightarrow \infty} \ldummy(n)\text{\;} \label{eq:lim_parameter}\\
	F^\text{ecm} &:= \lim_{n \rightarrow \infty} F^n \text{\;} \label{eq:lim_force}
\end{align}	
\caption{embedded cell algorithm for two dimensional linear hyperelastic isotropic materials with constant shear modulus (detailed)} 
\end{algorithm}
\vspace{3mm}

In numerical experiments, the limits $\kdummy$ and  $\ldummy$, respectively, and  $F^\text{ECM}$ have to be replaced by $\kdummy(N)$ and  $\ldummy(N)$, respectively, and  $F^N$ for a sufficiently large $N$ which has to be determined by an appropriate stop
criterion.

Now, let $\lambda^{\text{hom}}$ and $\kappa^{\text{hom}} =  \lambda^{\text{hom}} + 2 \mu^{\text{hom}}$, respectively, $u^{\text{hom}}$ as well as $F^{\text{hom}}$ be the result of the above analytical homogenization procedure. Then
the question of the correctness of the embedded cell method means: \\
\begin{enumerate}
	\item Do the limits \eqref{eq:lim_parameter_1d}, \eqref{eq:lim_force_1d}  and \eqref{eq:lim_parameter}, \eqref{eq:lim_force}, respectively exist?
	\item Do the identities
	\begin{align}
     \kdummy =  \kappa^{\text{hom}} \label{kcorrect}
     \end{align}
     and
      \begin{align}
     \ldummy =  \lambda^{\text{hom}}, \label{lcorrect}
     \end{align}	
	 respectively, and
	\begin{align}
		F^\text{ECM} =  F^{\text{hom}}
		\label{eq:force_equality}
	\end{align}
	hold?
	\end{enumerate}

\section{The one dimensional case}
\label{chap:correctness_1D}
It is the goal of this section to prove the correctness of the embedded cell method 
in one dimension. In this case, the proof can be performed by explicit calculations.
Let $ \Omega = (0,1)$, $l >0$ and $ \kappa $ be some step function.
Since we have no continuity of $ \kappa $ in $ \Omega $,  we have to consider weak 
solutions $u \in \calW_l$ to 
$$ 
\int_0^1 \kappa u' v' dx = 0 
$$ 
for all $ v \in \calW_0$.
Then the tensile force is given by 
$$ 
F= F[\kappa,l] = \kappa(1)  u'(1)
$$ 
in the trace sense.
We start with some periodic order of the materials and therefore consider $\kappa$ with
\begin{align}
{\kappa}(x) = \left\{ \begin{array}{cl}\kappam, & x \in [q,q+|\Omega_m|) ,\\
\kappac, & x \in [q+|\Omega_m|,q+1) \end{array} \right.
\end{align}
and  $q \in \mathbb{Z} $. Then we define $ \kappa_n(x) = {\kappa}(2^n x) $. 

For $ l = 0 $ the above boundary problem with $\kappa= \kappa_n$ is solved by $ u = u_n = 0 $.
For $ l > 0 $  and  $\kappa= \kappa_n$ the solution $u=u_n$ is determined by 
$$ 
(u_n)'(x) = \left\{ \begin{array}{cl} \alpha/\kappam  & \textrm{if } \kappa_n(x) = \kappam ,\\
\alpha/\kappac  & \textrm{if } \kappa_n(x) = \kappac \end{array} \right.
$$
since 
$$
\int_0^1 \kappa_n u' v' dx = \int_0^1 \alpha v' dx = \alpha  \int_0^1 v' dx = 0 
$$
for all $v \in \calW_0 $.
Hence we have independently of $ n $ that
$$ 
l = u_n(1)
= \alpha \left(\frac{|\Omega_m|}{\kappam}+ \frac{|\Omega_c|}{\kappac}\right)
$$ 
and so 
$$ 
\alpha
= l \left(\frac{|\Omega_m|}{\kappam}+ \frac{|\Omega_c|}{\kappac}\right)^{-1}.
$$ 
The corresponding tensile force is then given by
$$ 
F_n = \kappa_n(1) (u_n)'(1) = \left\{ \begin{array}{rl} \kappam \cdot \frac{\alpha}{\kappam} & \text{if} \,\, \kappa_n(1) = \kappam \\
\kappac \cdot \frac{\alpha}{\kappac} & \text{if} \,\, \kappa_n(1) = \kappac
\end{array}
\right\}
= l \left(\frac{|\Omega_m|}{\kappam}+ \frac{|\Omega_c|}{\kappac}\right)^{-1},
$$ 
and therefore also independent of $n$. 

Consequently, we obtain
\begin{align}
F^{\text{hom}} = l \left(\frac{|\Omega_m|}{\kappam}+ \frac{|\Omega_c|}{\kappac}\right)^{-1}
\end{align}
and
\begin{align}
\kappa^{\text{hom}} = \left(\frac{|\Omega_m|}{\kappam}+ \frac{|\Omega_c|}{\kappac}\right)^{-1}.
\end{align}

Now, we consider the embedded cell method.
Let
$$ 
{\kappa}^n(x) = \left\{ \begin{array}{cl}
\kappac, & x \in [1/2- |\Omega_c|/10,1/2+ |\Omega_c|/10) ,\\
\kappam, & x \in [2/5,1/2- |\Omega_c|/10) \cup [1/2+ |\Omega_c|/10,3/5)  ,\\
\kdummy(n), & x \in [0,2/5) \cup [3/5,1]. \end{array} \right.
$$
As above we find that the solution $u=u^n$ of the above boundary problem with $l>0$ and $\kappa=\kappa^n$ is determined by
$$ 
(u^n)'(x) = \left\{ \begin{array}{cl} \alpha/\kappac  & \textrm{if } \kappa^n(x) = \kappac, \\
 \alpha/\kappam  & \textrm{if } \kappa^n(x) = \kappam ,\\
\alpha/\kdummy(n)  & \textrm{if } \kappa^n(x) = \kdummy(n), \end{array} \right.
$$
and 
$$ 
l= u^n(1) = \frac{|\Omega_c|}{5} \frac{\alpha}{\kappac}+ \frac{|\Omega_m|}{5} \frac{\alpha}{\kappam}+ \frac{4}{5} \frac{\alpha}{\kdummy(n)}.
$$
Hence we have
$$ 
F^{n+1} = F[\kappa^n,l] = \kdummy(n) \frac{\alpha}{\kdummy(n)} = l \left(\frac{|\Omega_c|}{5 \kappac} +\frac{|\Omega_m|}{5 \kappam}+\frac{4 }{5 \kdummy(n)}\right)^{-1},
$$
and by \eqref{kequn} and \eqref{kdn} we get
\begin{align}
\label{kdit}
\kdummy(n+1) = \left(\frac{|\Omega_c|}{5 \kappac} + \frac{|\Omega_m|}{5 \kappam}+\frac{4 }{5 \kdummy(n)}\right)^{-1}.
\end{align}
Since the right hand side of \eqref{kdit} defines a contraction, there is a unique fixed point 
$ \kdummy  = \lim_{n \to \infty} \kdummy(n) $ satisfying 
$$ 
\kdummy = \left(\frac{|\Omega_c|}{5 \kappac} + \frac{|\Omega_m|}{5 \kappam}+\frac{4 }{5 \kdummy}\right)^{-1}
$$
and so 
$$ 
\kdummy  =  \left(\frac{|\Omega_c|}{\kappac} + \frac{|\Omega_m|}{ \kappam}\right)^{-1}.
$$ 
Finally, the tensile force $F^\text{ECM}$ obtained from the embedded cell method is given by 
$$ 
F^\text{ECM} = l \left(\frac{|\Omega_m|}{\kappam}+ \frac{|\Omega_c|}{\kappac}\right)^{-1},
$$ 
which is exactly the same as above.
Hence, we have proven

\begin{theorem}[correctness result]
	The sequence $(\kdummy(n))_{n \in \N}$ defined in algorithm \ref{alg:emc_detail_1d} converges to 
	\begin{align}\label{eq:limit_dummy_material}
		\kdummy = \left(\dfrac{\snorm{\Omega_m}}{\kappam} + \dfrac{\snorm{\Omega_c}}{\kappac}\right)^{-1}
	\end{align}
and the sequence $(F^n)_{n \in \N}$ defined in algorithm \ref{alg:emc_detail_1d} converges to
	\begin{align}\label{limit_force}
		F^\text{ECM} = l \left(\dfrac{\snorm{\Omega_m}}{\kappam} + \dfrac{\snorm{\Omega_c}}{\kappac}\right)^{-1} .
	\end{align}
Moreover, we have
\begin{align}
     \kdummy =  \kappa^{\text{hom}} 
     \end{align}
     and
     \begin{align}
		F^\text{ECM} =  F^{\text{hom}}.
		\end{align}
\end{theorem}

The calculations made above for periodic materials do not actually depend on the distribution, i.e. they also hold for any other material function $\kappa: \R \rightarrow \{\kappam, \kappac\}$ with the property that
\begin{align}
	\delta \snormleftright{\operatorname{supp}\{\kappa - \kappam\} \cap \left[0,\dfrac{1}{\delta}\right]} \xlongrightarrow{\delta \searrow 0} \snorm{\Omega_c}.
	\label{eq:property_material_func}
\end{align}

Before using this observation to prove a generalized homogenization result, the de\-finition of a random material function has to be given.

\begin{defi}[random material function]\label{def:random_material_function}
	Let $x \in Y := (0,1)$ and $\R$ be decomposed by
	\begin{align*}
		\R = \overline{\bigcup_{q \in \Z} Y_q}
	\end{align*}
	where $Y_q := q + x + Y$.
	The set of possible material functions $X$ includes all functions which are constant on the $Y_q$:
	\begin{align*}
		X := \{a : \R \rightarrow \{\kappam, \kappac\} \vert a\vert_{Y_q} \text{ is constant for all } q \in \Z \text{ for a}\; x \in Y\},
	\end{align*}
	where the probability that $a\vert_{Y_q} = \kappam$ is $\snorm{\Omega_m}$ and therefore, the probability that $a\vert_{Y_q} = \kappac$ is $\snorm{\Omega_c}$.
\end{defi}

There exists a probability measure $\mathcal{P}$ such that $(X, \mathcal{A}, \mathcal{P})$ is a probability space. 
Here, $\mathcal{A}$ is the $\sigma$-algebra of measurable subsets of $X$.
Using this framework and the law of large numbers the following theorem can be proven.

\begin{theorem}[stochastic homogenization]\label{thm:random_convergence}
	Let $(X, \mathcal{A}, \mathcal{P})$ be the probability space introduced above and for $\delta > 0$ and $\omega \in X$ let $F^{\delta}(\omega)$ be the tensile force belonging to the weak solution $u \in \calW_l$ of 
	\begin{align*}
		&\left(\omega\left(\dfrac{x}{\delta}\right) u'(x)\right)' = 0 \qquad \text{ in } \Omega.
	\end{align*}
	Then for almost all material functions $\omega$ in $X$ we have
	\begin{align}
		\lim_{\delta \rightarrow 0}  F^{\delta}(\omega) = l \left(\frac{|\Omega_m|}{\kappam}+ \frac{|\Omega_c|}{\kappac}\right)^{-1} =  F^{\text{hom}} = F^\text{ECM} .
		\label{eq:force_random_material}
	\end{align}
	\end{theorem}
\begin{proof}
	According to the remark above it is sufficient to show the property \eqref{eq:property_material_func}: 
	In order to do so, consider a sequence of Bernoulli random variables $(X_q)_{q \in \Z}$, i.e. the $X_q$ are random variables which assign only the two values 0 and 1 with probabilities $P[X_q = 0] = \snorm{\Omega_m}$ and $P[X_q = 1] = 1 - \snorm{\Omega_m} = \snorm{\Omega_c}$. 
	Then, every $\omega \in X$ can be written as
	\begin{align*}
		\omega = \kappam + (\kappac - \kappam) \sum_{q \in \Z} \chi_{Y_q}\, x_q 
	\end{align*}
	where $x_q$ is a realization of the random variable $X_q$.	
	Furthermore, the strong law of large numbers can be applied to the sequence $(X_q)_{q \in \Z}$ but also to the sequence $(X_n)_{n \in \N}$. This yields 
	\begin{align*}
		\lim_{n \rightarrow \infty} \dfrac{1}{n} \sum_{i = 0}^n x_n = \snorm{\Omega_c}
	\end{align*}
	for almost all realizations of the $X_n$ and since such a realization gives an element $\omega$ of the probability space $(X, \mathcal{A}, \mathcal{P})$, this means almost surely convergence in this probability space.
	Summarizing, this yields with $\delta := n^{-1}$ that
	\begin{align*}
		\lim_{\delta \rightarrow 0} \delta \snormleftright{\operatorname{supp}(\omega - \kappam) \cap \left[0, \dfrac{1}{\delta}\right]} &= \lim_{n \rightarrow \infty} \dfrac{1}{n} \snormleftright{\operatorname{supp}\left((\kappac - \kappam)\sum_{i \in \N_0} \chi_{Y_i} x_i\right) \cap [0,n]} \\
		&= \lim_{n \rightarrow \infty} \dfrac{1}{n} \snormleftright{\operatorname{supp}\left(\sum_{i \in \N_0} \chi_{Y_i} x_i\right) \cap [0,n]} \\
		&= \lim_{n \rightarrow \infty} \dfrac{1}{n} \snormleftright{\operatorname{supp}\left(\sum_{i =0}^n \chi_{Y_i} x_i\right)} \\
		&= \lim_{n \rightarrow \infty} \dfrac{1}{n} \sum_{i = 0}^n x_i = \snorm{\Omega_c}
	\end{align*}
	holds for almost every $\omega \in X$.
	Notice that this convergence also holds for any arbitrary sequence of $\delta$ tending to zero.
	Thus, \eqref{eq:property_material_func} holds almost surely in $(X, \mathcal{A}, \mathcal{P})$.
\end{proof}

We will finish this section by illustrating that the embedded cell method can also be applied to elasto-plastic materials. We will show for a simple 1D plasticity model
the existence of the limits and that the computed  limits give the correct
strain stress curves of the full problem.
Again the answer can be obtained by explicit calculations.
We extend the previous modeling of the metal
by choosing a simple nonlinear stress-strain curve for metal, namely
$$ 
F_{m}(l) = \left\{\begin{array}{cl} \alpha l, & \mbox{for  } l \in (0,l_0), \\
  \alpha l_0 + \beta (l-l_0)^{1/2}, & \mbox{for  } l \geq l_0, \end{array} \right.
$$ 
or equivalently 
$$ 
F_{m}(u_m') = \left\{\begin{array}{cl} \alpha u_m', & \mbox{for  } u_m'  \leq u_{m,crit}',  \\
  \alpha  u_{m,crit}'+ \beta (u_m'-u_{m,crit}')^{1/2}, & \mbox{for  } u_m'  > u_{m,crit}'. \end{array} \right.
$$ 
For simplicity, we consider a mixture $ 1/1 $ of metal and ceramics, and again we have a piecewise linear solution.
As a consequence  the length of the metal phase is $ u_m'/2 $ and of the ceramic phase $ u_c'/2 $.
Moreover, we can restrict ourselves to the case $ u_m'  > u_{m,crit}' $ since the  other case has already 
been handled.
Since the forces in both phases must be the same, we have 
$$ 
F = \kappac u_c' = \alpha u_{m,crit}' + \beta ((u_m'-u_{m,crit}')^{1/2}).
$$
From this we find 
\begin{eqnarray*}
u_c' & = & F/\kappac ,\\
u_m' & = & u_{m,crit}' + ( F - \alpha  u_{m,crit}')^2/\beta.
\end{eqnarray*}
Since $ l = (u_c' +u_m' )/2 $, we find the strain-stress relation 
\begin{equation} \label{ssr2}
\left( F/\kappac +  u_{m,crit}' + ( F - \alpha  u_{m,crit}')^2/\beta  \right)/2 = l .
\end{equation}
Because this relation holds for all material distributions with a mixture 1/1 of metal and ceramics, it remains valid if $F$ is replaced by  $F^{\text{hom}}$, where $F^{\text{hom}}$ is
defined as above.
 
For the embedded cell we use the same geometry as above.
 The length of the metal phase is now $ u_m'/10 $, of the ceramic phase $ u_c'/10 $, and 
 of the dummy material $ 4 u_d'/5 $.
Since the forces in all phases must be the same, we have 
$$ 
F = \kappa_c u_c' = \alpha u_{m,crit}' + \beta ((u_m'-u_{m,crit}')^{1/2}) =  \kappa_d u_d'
$$
From this we find 
\begin{eqnarray*}
u_c' & = & F/\kappac ,\\
u_m' & = & u_{m,crit}' + ( F - \alpha  u_{m,crit}')^2/\beta, \\
u_d' & = & F/\kappa_d .
\end{eqnarray*}
Since $ l = (u_c' +u_m' )/10 + 4 u_d' /5 $, we find the stress strain relation 
$$
\left( F/\kappa_c +  u_{m,crit}' + ( F - \alpha  u_{m,crit}')^2/\beta  \right)/10 + 4F /(5 \kappa_d) = l = F /\kappa^\text{equiv}.
$$
The iteration process for the computation of $ \kdummy $ is given by 
$$ 
\left( F/\kappac +  u_{m,crit}' + ( F - \alpha  u_{m,crit}')^2/\beta  \right)/10 + 4K /(5 \kdummy(n)) = l = F /\kdummy(n+1) .
$$
The mapping $ \kdummy(n)) \mapsto \kappa_{d,n+1} $ is a contraction, and so the limit 
$\kdummy  = \lim_{n \to \infty}$ $\kdummy(n)$ 
exists and $F^\text{ECM}$ also satisfies relation \eqref{ssr2}.

\section{Perturbation theory for the two dimensional case}\label{chap:perturbation_theory}
\subsection{Approximate solutions of the two dimensional model equations}\label{sec:perturbation_model}
Now, we address the question of the correctness of the embedded cell method in two dimensions
for metal-ceramic composite materials with constant shear modulus and slightly varying first Lamé parameter with the help of perturbation theory. The first step of our approach is to construct approximate solutions to the two dimensional model equations
\eqref{eq:DGL2D} - \eqref{eq:BC2D3} in the case of slightly varying $\lambda$ and $\mu$ and estimate the accuracy of these approximations.

We assume that $\lambda$ and $\mu$ are of the form
\begin{align}
\lambda(x) \;=\;	\lambda^\varepsilon(x) &= \lambda_0 + \varepsilon \lpert(x) ,\label{eq:lambda_pert}\\
\mu(x) \;=\;	\mu^\varepsilon(x) &= \mu_0 + \varepsilon \mpert(x) ,\label{eq:mu_pert}
\end{align}
with constants $\lambda_0, \mu_0 > 0$, $\lpert, \mpert \in L^{\infty}(\overline{\Omega})$, and
$\varepsilon \in (0,\varepsilon_0)$, where $\varepsilon_0$ is so small that 
$\lambda$ and $\mu$ satisfy \eqref{lbound} and \eqref{mbound} uniformly with respect to $\varepsilon$.

For the approximation, we make the ansatz
\begin{align}
	\uapprox^\varepsilon = u_0 + \varepsilon u_1.
	\label{ansatz}
\end{align}
Inserting \eqref{eq:lambda_pert} - \eqref{ansatz} into the weak formulation \eqref{eq:weak_solution} of \eqref{eq:DGL2D} - \eqref{eq:BC2D3} and equating the coefficients in front of the $\varepsilon^{m}$ with $m \in \{0,1\}$ yields that $u_0 \in \calW_l$ and $u_1 \in \calW_0$ are given as the unique solutions to
\begin{align}
	\int_\Omega \lambda_0 \trace(\nabla^s u_0) \trace(\nabla^s v) + 2\mu_0 \nabla^s u_0 : \nabla^s v dx &= 0 \label{eq:equation_for_u0}
\end{align}
for all $v \in \calW_0$, and
\begin{align}	
	\int_\Omega \lambda_0 \trace(\nabla^s u_1) \trace(\nabla^s v) + 2\mu_0 \nabla^s u_1 : \nabla^s v dx &= F_1(v)
\label{eq:equation_for_u1}
\end{align}	
for all $v \in \calW_0$, where
\begin{align}	
F_1(v)	&= - \int_\Omega \lpert \trace(\nabla^s u_0) \trace(\nabla^s v) + 2\mpert \nabla^s u_0 : \nabla^s v dx.
\end{align}
Notice that both weak solutions exist since the assumptions of theorem \ref{thm:existence_uniqueness} are satisfied.

Because $\lambda_0, \mu_0$ are constant real numbers, the solution $u_0$ can be explicitly computed by inserting the ansatz 
\begin{align*}
	u_0 (x)= \left(\begin{array}{cc}
		a_1 & 0 \\
		0 & a_2
	\end{array}\right) x + \left(\begin{array}{c}
		b_1 \\
		b_2
	\end{array}\right)
\end{align*}
into \eqref{eq:equation_for_u0}, using Green's formula as well as the fundamental lemma of calculus of variations and taking into account that $u_0 \in \calW_l$, which yields
\begin{align}
	u_0 (x)= \left(\begin{array}{cc}
		-\nu_0 l & 0 \\
		0 & l
	\end{array}\right) x + \left(\begin{array}{c}
		\frac{1}{2} \nu_0 l \\
		0
	\end{array}\right), \qquad \nu_0 = \dfrac{\lambda_0}{\lambda_0 + 2\mu_0},
	\label{eq:u_0}
\end{align}
where $\nu_0$ is the so-called Poisson number.

Concerning the accuracy of the approximation $\uapprox$ we have the following theorem.
\begin{theorem}[approximation result]\label{thm:approximation_result_u}
	Let $l \in \R$ be given, $u^{\varepsilon}\in \calW_l$ be the unique weak solution of \eqref{eq:DGL2D} - \eqref{eq:BC2D3}, where $\lambda$ and $\mu$ satisfy \eqref{eq:lambda_pert} -   
\eqref{eq:mu_pert}, and $\uapprox^\varepsilon \in \calW_l$ be given by \eqref{ansatz} - \eqref{eq:equation_for_u1}. Then, there exists a constant $C>0$ independent of $\varepsilon > 0$ such that
	\begin{align} \label{estR}
		\|u^\varepsilon - \uapprox^\varepsilon\|_{H^1(\Omega)} \leq C\varepsilon^2.
	\end{align}
\end{theorem}
\begin{proof}
By construction of $\uapprox^\varepsilon$ we have $\uapprox^\varepsilon \in \calW_l$ and 
\begin{align*}
	\int_\Omega \lambda^\varepsilon \trace(\nabla^s \uapprox^\varepsilon) \trace(\nabla^s v) + 2\mu^\varepsilon \nabla^s \uapprox^\varepsilon : \nabla^s v dx &= - \widetilde{F}(v)
\end{align*}
for all $v \in \calW_0$, where	
\begin{align*}	
\widetilde{F}(v)	&= -\varepsilon^2 \int_\Omega \lpert \trace(\nabla^s u_1) \trace(\nabla^s v) + 2\mpert \nabla^s u_1 : \nabla^s v dx .
\end{align*}
Therefore, $R:= u^\varepsilon - \uapprox^\varepsilon \in \calW_0$ solves
\begin{align*}
	&\int_\Omega \lambda^\varepsilon \trace(\nabla^s R) \trace(\nabla^s v) + 2\mu^\varepsilon \nabla^s R : \nabla^s v dx = \widetilde{F}(v)
\end{align*}
for all $v \in \calW_0$.
Since all assumptions of the existence and uniqueness theorem \ref{thm:existence_uniqueness} are satisfied, we obtain
\begin{align}
\|R\|_{H^{1}(\Omega)} &\leq C_1 \|\widetilde{F}\|_{\calW_0'} \nonumber \\
&= C_1 \sup_{\substack{v \in \calW_0, \\  \norm[H^1(\Omega)]{v} = 1}} \Big| \varepsilon^2 \int_\Omega \lpert \trace(\nabla^s u_1) \trace(\nabla^s v) + 2\mpert \nabla^s u_1 : \nabla^s v dx \Big| \nonumber \\
&\leq \varepsilon^2 C_2 (\|\lpert\|_{L^{\infty}(\Omega)} + 2 \|\mpert\|_{L^{\infty}(\Omega)} )\norm[H^1(\Omega)]{u_1} \label{estRtF}  
\end{align}
for constants $C_1,C_2>0$. Moreover, because $u_1 \in  \calW_0$ solves \eqref{eq:equation_for_u1}, we can use again theorem \ref{thm:existence_uniqueness} as well as \eqref{eq:u_0} to get
\begin{align}
\|u_1\|_{H^{1}(\Omega)} &\leq C_3 \|F_1\|_{\calW_0'} \,\leq\, C_4  \label{estu_1}
\end{align}
for constants $C_3,C_4>0$. 

Now, combining \eqref{estRtF} and \eqref{estu_1} yields \eqref{estR}.
\end{proof}

Hence, the relative error of the $\uapprox^\varepsilon$ is small such that it is reasonable to use the approximation $\uapprox^\varepsilon$  instead of the exact solution $u^\varepsilon$.

\subsection{Approximate solutions of the homogenization problem}
\label{sec:apphom}
Next, we construct approximate solutions to the homogenization problem \eqref{anahom} and estimate the accuracy of these approximations. The main advantage of our approximate solutions will be that they can be explicitly computed.

We consider the family of perturbed material parameters
\begin{align}
	\lambda^{\varepsilon,\delta}(x) &= \lambda_0 + \varepsilon \lpert\left(\dfrac{x}{\delta}\right) =: \lambda_0 + \varepsilon \lpert^\delta(x), \\
	\mu^{\varepsilon,\delta}(x) &= \mu_0 + \varepsilon \mpert\left(\dfrac{x}{\delta}\right) =: \mu_0 + \varepsilon \mpert^\delta(x),
\end{align}
with $\delta > 0$, periodic functions $\lpert, \mpert \in L^{\infty}(\R^2)$ and $\lambda_0$, $\mu_0$ and $\varepsilon$ as in \eqref{eq:lambda_pert}-\eqref{eq:mu_pert}.
Then, according to theorem \ref{thm:approximation_result_u} the family of solutions $(u^{\varepsilon,\delta})_{\delta > 0} \subset \calW_l$ to
\begin{align}
	\int_\Omega \lambda^{\varepsilon,\delta} \trace(\nabla^s v) \trace(\nabla^s u^{\varepsilon,\delta}) + 2 \mu^{\varepsilon,\delta} \nabla^s v : \nabla^s u^{\varepsilon,\delta} dx = 0
	\label{eq:homogenization_equation_full}
\end{align}
for all $v \in \calW_0$ is of the form
\begin{align}
	u^{\varepsilon,\delta} = u_0 + \varepsilon u_1^\delta + \calO(\varepsilon^2),
	\label{eq:approx_homogenization}
\end{align}
where $u_0$ is given by \eqref{eq:u_0} and $u_1^\delta \in \calW_0$ solves 
\begin{align*}
	&\int_\Omega \lambda_0 \trace(\nabla^s u^\delta_1) \trace(\nabla^s v) + 2\mu_0 \nabla^s u^\delta_1 : \nabla^s v dx\\
& \qquad \quad = - \int_\Omega \lpert^\delta \trace(\nabla^s u_0) \trace(\nabla^s v) + 2\mpert^\delta \nabla^s u_0 : \nabla^s v dx
\end{align*}
for all $v \in \calW_0$.

According to the homogenization theorem \ref{thm:homogenization_convergence_result}, there exists a unique $u^{\varepsilon,\text{hom}} \in \calW_l$ and unique constants $\lambda^{\varepsilon,\text{hom}}, \mu^{\varepsilon,\text{hom}} \in \R$ such that 
	\begin{align} \label{eff1}
		\begin{cases}
			u^{\varepsilon,\delta} \rightharpoonup u^{\varepsilon,\text{hom}} & \text{ in } \calW_l \\
			\lambda^{\varepsilon,\delta} \trace(\nabla^s u^{\varepsilon,\delta}) I + 2 \mu^{\varepsilon,\delta} \nabla^s u^{\varepsilon,\delta} \rightharpoonup \lambda^{\varepsilon,\text{hom}} \trace(\nabla^s u^{\varepsilon,\text{hom}}) I + 2\mu^{\varepsilon,\text{hom}} \nabla^s u^{\varepsilon,\text{hom}} & \text{ in } L^2(\Omega, \R^{2\times 2})
		\end{cases}
	\end{align}
for $\delta \rightarrow 0$	and
	\begin{align} \label{eff2}
		\int_\Omega \lambda^{\varepsilon,\text{hom}} \trace(\nabla^s u^{\varepsilon,\text{hom}}) \trace(\nabla^s v) + 2\mu^{\varepsilon,\text{hom}} \nabla^s u^{\varepsilon,\text{hom}} : \nabla^s v dx = 0
	\end{align}
	for all $v \in \calW_0$.

To construct an approximation for $u^{\varepsilon,\text{hom}}$ we consider the family 
$(u_\text{approx}^{\varepsilon,\delta})_{\delta > 0} \subset \calW_l$ with $u^{\varepsilon,\delta}_\text{approx} := u_0 + \varepsilon u_1^{\delta}$.  By construction of $\lambda^{\varepsilon,\delta}$ and $\mu^{\varepsilon,\delta}$, estimate \eqref{ellest} implies that $(u_\text{approx}^{\varepsilon,\delta})_{\delta > 0}$ is uniformly bounded with respect to $\delta$ in $H^1(\Omega)$.  Because $H^1(\Omega)$ is a Hilbert space and  $\calW_l$ is closed in $H^1(\Omega)$, there exists a $u^{\varepsilon,\text{hom}}_\text{approx} \in \calW_l$ and a sequence  $(\delta_n)_{n \in \N}$  with $\delta_{n} \rightarrow 0$ for $n \rightarrow \infty$ such that
\begin{align}
\label{u1hseq}
	u^{\varepsilon,\delta_{n}}_\text{approx}  \rightharpoonup u^{\varepsilon, \text{hom}}_\text{approx} =: u_0 + \varepsilon u_1^{\text{hom}}
\end{align}
in $H^1(\Omega)$ for $n \rightarrow \infty$.

It is possible to compute $\uapprox^{\varepsilon,\text{hom}}$ explicitly.
The component $u_0$ is explicitly given by \eqref{eq:u_0}. To compute $u_1^{\text{hom}}$, we
use that the functions $u_1^{\delta_{n}}$ from \eqref{u1hseq} solve
\begin{align*}
	B(u_1^{\delta_{n}},v) = f(\lpert^{\delta_{n}},\mpert^{\delta_{n}},v)  
\end{align*}
for all $v \in \calW_0$, where 
\begin{align*}
	B(u, v) &= \int_\Omega \lambda_0 \trace(\nabla^s u) \trace(\nabla^s v) + 2\mu_0 \nabla^s u : \nabla^s v dx, \\
	f(\lambda,\mu, v) &= -\int_\Omega \lambda \trace(\nabla^s u_0) \trace(\nabla^s v) + 2\mu \nabla^s u_0 : \nabla^s v dx.
\end{align*}
Since $\lpert$ and $\mpert$ are assumed to be periodic it holds 
\begin{align}
	&\lpert^{\delta_{n}} \rightharpoonup \bar{\lambda} := \dfrac{1}{\snorm{\Omega}} \int_\Omega \lpert dx  \label{lhom}\\
	&\mpert^{\delta_{n}} \rightharpoonup \bar{\mu} := \dfrac{1}{\snorm{\Omega}} \int_\Omega \mpert dx \label{mhom}
\end{align}
in $L^2(\Omega)$ for $\delta_{n} \rightarrow 0$. Because of \eqref{u1hseq}, \eqref{lhom} and \eqref{mhom} we obtain
\begin{align*}
	B(u_1^\text{hom},v) = f(\bar{\lambda},\bar{\mu},v)  
\end{align*}
for all $v \in \calW_0$. 

Hence, $u_1^\text{hom} \in \calW_0$ is the (according to theorem \ref{thm:existence_uniqueness} unique) solution to
\begin{align}
	&\int_\Omega \lambda_0 \trace(\nabla^s u_1^\text{hom}) \trace(\nabla^s v) + 2\mu_0 \nabla^s u_1^\text{hom} : \nabla^s v dx \nonumber \\ 
	&\qquad\quad = - \int_\Omega \bar{\lambda} \trace(\nabla^s u_0) \trace(\nabla^s v) + 2\bar{\mu} \nabla^s u_0 : \nabla^s v dx
	\label{eq:limit_u_1_hom}
\end{align}
for all $v \in \calW_0$.

Because $\bar{\lambda}, \bar{\mu}$ are constant real numbers, the solution $u_1^\text{hom}$ can be explicitly computed by inserting the ansatz 
\begin{align*}
	u_1^\text{hom}(x) = \left(\begin{array}{cc}
		a_1 & 0 \\
		0 & a_2
	\end{array}\right) x + \left(\begin{array}{c}
		b_1 \\
		b_2
	\end{array}\right)
\end{align*}
into \eqref{eq:limit_u_1_hom}, using Green's formula as well as the fundamental lemma of calculus of variations and taking into account that $u_1^\text{hom} \in \calW_0$, which yields
\begin{align}
	u_1^\text{hom}(x) = \dfrac{\bar{\lambda} (1-\nu_0) l - 2\bar{\mu} \nu_0 l}{2(\lambda_0 + 2 \mu_0)} \left(\begin{array}{cc}
		-2 & 0 \\
		0 & 0
	\end{array}\right) x + \left(\begin{array}{c}
		1 \\ 0
	\end{array}\right).
	\label{eq:u_1_hom}
\end{align}
To compute $u_1^\text{hom}$ we have used that $\delta_{n} \rightarrow 0$ but not the special values of  the sequence $(\delta_{n})_{n \in \N}$  such that we have  $u^{\varepsilon,\delta_{n}}_\text{approx}  \rightharpoonup u^{\varepsilon, \text{hom}}_\text{approx} $ for any sequence  $(\delta_{n})_{n \in \N}$ with  $\delta_{n} \rightarrow 0$.

Now, an interesting question is if an analogous result to theorem \ref{thm:approximation_result_u} can be proven for $u^{\varepsilon, \text{hom}}_\text{approx} $.
In fact, such a result holds as the following theorem shows.

\begin{theorem}[approximation result]\label{thm:approximation_result_homogenization}
	Let $u^{\varepsilon,\text{hom}}$ and $\uapprox^{\varepsilon,\text{hom}}$ be given as the weak limits of $(u^{\varepsilon,\delta})$ and $(u^{\varepsilon,\delta}_\text{approx})$, respectively, for $\delta \to 0$. 
	Then there exists a $C > 0$ independently of $\varepsilon$ and $\delta$ such that
	\begin{align*}
		\norm[H^1(\Omega)]{u^{\varepsilon, \text{hom}} - u^{\varepsilon, \text{hom}}_\text{approx}} \leq C \varepsilon^2
	\end{align*}
\end{theorem} 
\begin{proof}
	Due to theorem \ref{thm:approximation_result_u} there exists a constant $C > 0$ independent of $\delta$ such that
	\begin{align*}
		\norm[H^1(\Omega)]{u^{\varepsilon, \delta} - u^{\varepsilon, \delta}_\text{approx}} \leq C \varepsilon^2
	\end{align*}
	since the constant $C$ from theorem \ref{thm:approximation_result_u} can be bounded independently of $\delta$. Because the $H^1$-norm is lower semi-continuous with respect to weakly convergent sequences, the assertion of the theorem follows.
\end{proof}

From now on, a few simplifications are made, which we have also used in the ana\-lysis of the embedded cell method in section \ref{chap:ecm}.
\begin{itemize}
	\item The material is supposed to have constant shear modulus $\mu = \mu_0$.
	\item We consider a metal-ceramic composite material with first Lamé parameter
	\begin{align}
		\lambda^\varepsilon(x) = \lambdam + \varepsilon D_c \chi_{\Omega_c}(x)
		\label{eq:lambda_metal_ceramic}
	\end{align}
	where the constants $\lambdam = \lambda_0$ and $\lambdac>0$ are the first Lamé parameters of the metal and ceramics, respectively, $\Omega_c$ is the area occupied by the ceramic particles, and $\varepsilon D_c := \lambdac - \lambdam$.
	\item The sequence $\lpert^\delta$ is defined by periodic continuation of $D_c \chi_{\Omega_c}$.
\end{itemize}

Our next goal is to approximate the effective material parameter $\lambda^{\varepsilon,\text{hom}}$  up to an error of order $\varepsilon^2$. We obtain the following approximation result.

\begin{theorem}[approximation of the effective material parameter]\label{thm:approximation_result}
	Let $l \in \R$ and $(\lambda^{\varepsilon,\delta})_{\delta > 0}$ $\subset L^\infty(\Omega)$ be a family of material parameters defined as above.
Then there exists an $\varepsilon_0 > 0$ such that the effective material parameter defined according to theorem \ref{thm:homogenization_convergence_result} is given by
	\begin{align}
		\lambda^{\varepsilon,\text{hom}} = \lambdam + \varepsilon \snorm{\Omega_c} D_c + \calO(\varepsilon^2)
		\label{eq:approximation_hom_materialParam}
	\end{align}
	for all $\varepsilon \in (0, \varepsilon_0)$.
\end{theorem}

\begin{proof}
The strategy to obtain an approximation of the effective material parameter is to use the corresponding tensile force $F$. Using theorem \ref{thm:tensile_force} and  \eqref{eq:approx_homogenization} we obtain

\begin{align*}
F[\lambda^{\varepsilon,\delta}, \mu, l] &= \int_\Omega \lambda^{\varepsilon,\delta} \trace(\nabla^s u^{\varepsilon,\delta}) + 2\mu \partial_2 (u^{\varepsilon,\delta})_2 dx \\
&= \int_\Omega \lambdam \trace(\nabla^s u_0) + 2\mu \partial_2 (u_0)_2 dx \\
	&\qquad\qquad\qquad + \varepsilon\left[\int_\Omega \lpert^\delta \trace(\nabla^s u_0) + \lambdam \trace(\nabla^s u_1^\delta) + 2\mu \partial_2 (u_1^\delta)_2 dx\right] + \calO(\varepsilon^2) \\
	&=: F_0 + \varepsilon F^\delta_1 + \calO(\varepsilon^2).
\end{align*}

Because of \eqref{eq:u_0}, \eqref{u1hseq}, \eqref{lhom}, where in the present situation we have $\bar{\lambda} = \snorm{\Omega_c} D_c$, and \eqref{eq:u_1_hom} we get
\begin{align*}
F_0 = ((1-\nu_0) \lambdam + 2\mu) l 
\end{align*}
and
\begin{align*}
\lim_{\delta \rightarrow 0} F^\delta_1 &= \int_\Omega \bar{\lambda} \trace(\nabla^s u_0) + \lambdam \trace(\nabla^s u_1^\text{hom}) + 2\mu \partial_2 (u_1^\text{hom})_2 dx\\
	&= \snorm{\Omega_c} D_c (1-\nu_0)^2 l.
	\end{align*}
Hence, due to \eqref{eff1}, this implies 
\begin{align}
	F[\lambda^{\varepsilon,\text{hom}}, \mu, l]) &= ((1-\nu_0) \lambdam + 2\mu) l  + \varepsilon \snorm{\Omega_c} D_c (1-\nu_0)^2 l + \calO(\varepsilon^2).
	\label{eq:force_approximation_hom_1}
\end{align}
On the other hand, by using \eqref{eff2} we obtain
\begin{align}
	F[\lambda^{\varepsilon,\text{hom}}, \mu, l]) 
&= \int_\Omega \lambda^{\varepsilon,\text{hom}} \trace(\nabla^s u^{\varepsilon,\text{hom}}) + 2\mu \partial_2 (u^{\varepsilon,\text{hom}})_2 dx  \nonumber\\	
&= (1-\nu_0) l \lambda^{\varepsilon,\text{hom}} + 2\mu l  - \varepsilon \frac{\snorm{\Omega_c} D_c (1-\nu_0) \nu_0 l}{\lambdam}  \lambda^{\varepsilon,\text{hom}} + \calO(\varepsilon^2).
	\label{eq:force_approximation_hom_2}
\end{align}

Now, combining \eqref{eq:force_approximation_hom_1} and \eqref{eq:force_approximation_hom_2}
yields
\begin{align*}
	\lambda^{\varepsilon,\text{hom}} &=  \lambdam + \varepsilon \snorm{\Omega_c} D_c + \calO(\varepsilon^2)
\end{align*}
if $\varepsilon_0$ is chosen sufficiently small.
\end{proof}

\subsection{Correctness of the embedded cell method}
In this subsection, we prove a correctness result for the embedded cell method in two dimensions applied to metal-ceramic composite materials with constant shear modulus $\mu=\mu_0$ and slightly varying first Lamé parameter $\lambda$. More precisely, 
we assume as in the previous subsection that for  $0 < \varepsilon < \varepsilon_0$ with sufficiently small $\varepsilon_0$ the material parameter $\lambda$ in the metal phase $\Omega_m$ is given by $\lambdam=\lambda_0$ and in the ceramic phase $\Omega_c$ by
$\lambda^\varepsilon_{\text cer}=\lambdam + \varepsilon D_c$ with some $D_c \geq 0$.
We show that under these assumptions the iteration sequence $(\lepsdummy(n))_{n \in \N}$ defined by algorithm \ref{alg:emc_detail} is monotone, bounded and hence convergent and the limit 
$\lepsdummy$ satisfies
\begin{align}
	\lepsdummy = \lambdam + \varepsilon \snorm{\Omega_c} D_c + \calO(\varepsilon^2)
\end{align}
for all $\varepsilon \in (0, \varepsilon_0)$ such that 
\begin{align}
	\lambda^{\varepsilon,\text{hom}} - \lepsdummy = \calO(\varepsilon^2)
\end{align}
and consequently
\begin{align}
	F[\lambda^{\varepsilon,\text{hom}},\mu,l] - F[\lepsdummy,\mu,l] = \calO(\varepsilon^2)
\end{align}
for all $\varepsilon \in (0, \varepsilon_0)$.

The proof relies on the monotony of the tensile force with respect to the first Lamé parameter $\lambda$, see the subsequent lemma \ref{thm:monotonicity_force}, and a representation formula for $\lepsdummy(n)$, see the subsequent lemma \ref{lem:uniform_boundary_lambdaNM}.

\begin{lem}[monotony of the tensile force]\label{thm:monotonicity_force}
	Let $\lambda^j \in L^\infty(\Omega), j = 1,2$, be some functions for the first Lamé parameter such that there exists an $n \in \N$ such that
	\begin{align*}
		\lambda^j = \sum_{k=0}^n \varepsilon^k \lambda^j_k + \calO(\varepsilon^{n+1})
	\end{align*}
	with $\lambda^1_0 = \lambda^2_0 \equiv \lambda_0 \in \R$,
	\begin{align*}
		\lambda^2 - \lambda^1 = \varepsilon^n (\lambda_n^2 - \lambda^1_n) + \calO(\varepsilon^{n+1}),
	\end{align*}
	and
	\begin{align*}
		\int_\Omega \lambda_n^2 - \lambda_n^1 dx > 0.
	\end{align*}
	Then there exists an $\varepsilon_0 > 0$ such that
	\begin{align*}
		F[\lambda^1, \mu, l]  < F[\lambda^2, \mu, l]
	\end{align*}
	for all $l > 0$ and all $\varepsilon \in (0,\varepsilon_0)$.
\end{lem}
\begin{proof}
According to the generalized approximation theorem \ref{thm:generalized_approximation_u} the solutions $u^j \in \calW_l$, $j = 1,2$, of \eqref{eq:weak_solution} with the first Lamé parameter $\lambda^j$ are given by
\begin{align*}
	u^j = \sum_{k = 0}^n \varepsilon^k u_k^j + \calO(\varepsilon^{n+1})
\end{align*}
with
\begin{align}
	\int_0 \lambda_0 \trace(\nabla^s u_k^j) \trace(\nabla^s v) + 2\mu \nabla^s u_k^j : \nabla^s v dx = - \sum_{i = 0}^{k-1} \int_\Omega \lambda^j_{k-i} \trace(\nabla^s u_i^j) \trace(\nabla^s v) dx
	\label{eq:PDE_for_u_k}
\end{align}
for all $v \in \calW_0$ and $k = 0, \dots, n$.
Because of $\lambda^1_k = \lambda^2_k$ for all $k = 0,\dots, n-1$
we obtain inductively that $u^1_k = u^2_k$ for all $k = 0,\dots, n-1$ since the right-hand-side of \eqref{eq:PDE_for_u_k} is independent of $j$.

Therefore, the corresponding tensile forces have the representation
\begin{align*}
	F^j = F[\lambda^j, \mu, l] = \sum_{k=0}^n \varepsilon^k F_k^j + \calO(\varepsilon^{n+1}) + 2\mu l
\end{align*}
with
\begin{align*}
	F^j_k = \sum_{i = 0}^k \int_\Omega \lambda_{k-i} \trace(\nabla^s u_i^j) dx,
\end{align*}
which follows from a direct calculation.
Since $\lambda^j_k, u^j_k$ is not depending on $j$ for $k = 0, \dots, n-1$, this yields
\begin{align}
	F^2 - F^1 &= \varepsilon^n \left(F^2_n - F^1_n\right) + \calO(\varepsilon^{n+1}) \nonumber\\
	&= \varepsilon^n \left(\int_\Omega \lambda_0 \trace(\nabla^s(u^2_n - u^1_n)) dx + \int_\Omega (\lambda^2_n - \lambda^1_n) \trace(\nabla^s u_0) dx\right) + \calO(\varepsilon^{n+1}) \nonumber\\
	&= \varepsilon^n \left(\int_\Omega \lambda_0 \trace(\nabla^s(u^2_n - u^1_n)) dx + (1-\nu_0) l \int_\Omega (\lambda^2_n - \lambda^1_n) dx\right) + \calO(\varepsilon^{n+1}).
	\label{eq:Force_difference}
\end{align}
Moreover, $w := u_n^2 - u_n^1 \in \calW_0$ satisfies
\begin{align*}
	\int_\Omega \lambda_0 \trace(\nabla^s w) \trace(\nabla^s v) + 2\mu \nabla^s w : \nabla^s v dx = - \int_\Omega (\lambda^2_n - \lambda_n^1) \trace(\nabla^s u_0) \trace(\nabla^s v) dx
\end{align*}
for all $v \in \calW_0$.
Hence, choosing $v = (x_1,0)^T \in \calW_0$ and using the subsequent lemma \ref{lem:vanishing_mean_in_W0} we get
\begin{align*}
	(\lambda_0 + 2\mu) \int_\Omega \trace(\nabla^s w) dx = - (1-\nu_0) l \int_\Omega \lambda_n^2 - \lambda_n^1 dx.
\end{align*}
Furthermore, inserting this into \eqref{eq:Force_difference} yields
\begin{align*}
	F^2 - F^1 &= \varepsilon^n \left((1- \nu_0) l \int_\Omega \lambda_n^2 - \lambda_n^1 dx - \dfrac{\lambda_0}{\lambda_0 + 2\mu} (1- \nu_0) l \int_\Omega \lambda_n^2 - \lambda_n^1 dx\right) + \calO(\varepsilon^{n+1}) \\
	&= \varepsilon^n (1-\nu_0)^2 l \int_\Omega \lambda_n^2 - \lambda_n^1 dx + \calO(\varepsilon^{n+1}).
\end{align*}
Therefore, since 
$$\int_\Omega \lambda_n^2 - \lambda_n^1 dx > 0$$ 
and $l > 0$ according to the assertion, there exists an $\varepsilon_0 > 0$ such that $F^2 - F^1 > 0$ for all $\varepsilon \in (0,\varepsilon_0)$.
\end{proof}

It remains to prove the lemma mentioned in the previous proof.

\begin{lem}\label{lem:vanishing_mean_in_W0}
	Let $u \in \calW_0$ be arbitrarily chosen. Then it holds that
	\begin{align} \label{eq:vanishing_mean}
		\int_\Omega \partial_2 u_2 dx = 0
	\end{align}
\end{lem}
\begin{proof}
First, note that $f: H^1(\Omega, \R) \rightarrow \R$ defined by
\begin{align*}
	f[u] := \int_\Omega \partial_2 u dx
\end{align*}
is bounded in $H^1(\Omega)$.
Hence, to prove the assertion it is sufficient to prove that $f \equiv 0$ on the dense subset     $C^\infty_{\Gamma_1}(\bar{\Omega})$ of $H^1_{\Gamma_1}(\Omega)$ since $H^1_{\Gamma_1}(\Omega)$ is closed in $H^1(\Omega)$.
Now, for arbitrary $u \in C^\infty_{\Gamma_1}(\bar{\Omega})$ we have 
\begin{align*}
	f[u] = \int_\Omega \partial_2 u dx = \int_0^1 \int_0^1 \partial_2 u dx_2 dx_1 = \int_0^1 u(x_1, 1) - u(x_1, 0) dx_1 = 0,
\end{align*}
which proves the assertion of the lemma because $u \in \calW_0$ implies $u_2 \in H^1_{\Gamma_1}(\Omega)$.
\end{proof}

\begin{lem}\label{lem:uniform_boundary_lambdaNM}
	Suppose that $\lepsdummy(1) \neq \lepsdummy(0)$.
	Then there exist $\varepsilon_0 > 0$ and $m \in \N$, $m \geq 2$, such that
	\begin{align}
		\lepsdummy(n) = \lambdam + \varepsilon \snorm{\Omega_c} D_c + \varepsilon^m \lambda_{m,n} + \varepsilon^{m+1} R_{\varepsilon,n}
		\label{eq:representation_lepsdummy}
	\end{align}
	for all $\varepsilon \in (0,\varepsilon_0)$, where 
\begin{align}
\label{neq}
\lambda_{m,n+1} \neq \lambda_{m,n}
\end{align}
for all $n \in \N$ and
$\lambda_{m,n},\, R_{\varepsilon,n} = \calO(1)$ uniformly for all $n \in \N$ and all $\varepsilon \in (0,\varepsilon_0)$.
\end{lem}
\begin{proof}
The proof consists of two steps. The first step is to show by complete induction using the iteration procedure from algorithm \ref{alg:emc_detail} that $\lepsdummy(n)$ possesses a re\-presentation of the form \eqref{eq:representation_lepsdummy} which satisfies \eqref{neq}, and the second one is to prove that $\lambda_{m,n}, R_{\varepsilon,n} = \calO(1)$ uniformly for all $n \in \N$ and all $\varepsilon \in (0,\varepsilon_0)$.
We split the first step into the formulation and the proof of the subsequent lemmas \ref{thm:approximation_lambda_dummy}-\ref{lem:induction_part2}.

\begin{lem}\label{thm:approximation_lambda_dummy}
For all $n \in \N$ there exists an $\varepsilon_0 > 0$ such that
	\begin{align}
		\lepsdummy(n) = \lambdam + \varepsilon \snorm{\Omega_c} D_c + \calO(\varepsilon^2)
		\label{eq:approximation_dummy_parameter}
	\end{align}
	for all $\varepsilon \in (0,\varepsilon_0)$.
\end{lem}
\begin{proof}
The assertion is proven inductively.
For $n = 0$ we use \eqref{ld0} to obtain
\begin{align*}
	\lepsdummy(0) &= \snorm{\Omega_c} \lambdac^\varepsilon + \snorm{\Omega_m} \lambdam = \lambdam + \varepsilon \snorm{\Omega_c} D_c,
\end{align*}
such that \eqref{eq:approximation_dummy_parameter} is satisfied for $n = 0$.

Now, let \eqref{eq:approximation_dummy_parameter} be valid for some $n \in \N$.
Let $u^n \in \calW_l$ be the solution of \eqref{eq:weak_form_ECM} with the first Lamé parameter $\lambda_n^\varepsilon$ given by
\begin{align*}
	\lambda_n^\varepsilon =  \chi_{\widetilde{\Omega}_c} \lambdac^\varepsilon + \chi_{\widetilde{\Omega}_m} \lambdam + \chi_{\widetilde{\Omega}_\text{dummy}} \lepsdummy(n) =: \lambdam + \varepsilon \lambda_{n,\,\text{pert}}.
\end{align*}
Then, according to theorem \ref{thm:approximation_result_u}, the solution $u^n$ can be approximated by
\begin{align*}
	u^n = u_0 + \varepsilon u_1^n + \calO(\varepsilon^2)
\end{align*}
with $u_0 \in \calW_l$ given by \eqref{eq:u_0} and $u_1^n \in \calW_0$ given by the solution of \eqref{eq:equation_for_u1}.
Hence, as in the preceding subsection, an approximation for the corresponding tensile force is obtained by
\begin{align*}
	F[\lambda_n^\varepsilon, \mu, l] &= \int_\Omega \lambda^\varepsilon_n \trace(\nabla^s u^n) + 2 \mu \partial_2 (u^n)_2 dx \\
	&= \int_\Omega \lambdam \trace(\nabla^s u_0) + 2\mu \partial_2 (u_0)_2 dx \\
	&\qquad\qquad\qquad+ \varepsilon \left[\int_\Omega \lambda_{n,\,\text{pert}} \trace(\nabla^s u_0) + \lambdam \trace(\nabla^s u_1^n) + 2\mu \partial_2 (u_1^n)_2 dx\right] + \calO(\varepsilon^2) \\
	&=: F_0 + \varepsilon F_1^n + \calO(\varepsilon^2).
\end{align*}
We have
\begin{align}
F_0 = ((1-\nu_0) \lambda_0 + 2\mu) l . \label{F0lem}
\end{align}

In order to compute $F_1^n$, we insert $v = (x_1,0)^T \in \calW_0$ into \eqref{eq:equation_for_u1} and use lemma \ref{lem:vanishing_mean_in_W0} and $\mpert = 0$ to get
\begin{align*}
	\int_\Omega \lambdam \trace(\nabla^s u_1^n) + 2\mu \partial_1 (u_1^n)_1 dx &= -\int_\Omega \lambda_{n,\,\text{pert}} \trace(\nabla^s u_0) dx, 
\end{align*}	
which yields	
\begin{align*}	
\int_\Omega \partial_1 (u_1^n)_1 dx &= -\dfrac{(1-\nu_0) l}{\lambdam + 2\mu} \int_{\Omega} \lambda_{n,\,\text{pert}} dx.
\end{align*}
Therefore, we obtain 
\begin{align}
	F_1^n &=  (1-\nu_0)^2 l \int_\Omega \lambda_{n,\,\text{pert}} dx \nonumber \\[2mm]
	&= (1-\nu_0)^2 l D_c \left (\snorm{\widetilde{\Omega}_c}  + \snorm{\widetilde{\Omega}_\text{dummy}} \snorm{\Omega_c}  + \calO(\varepsilon)\right ) \nonumber \\[2mm]
	&= (1-\nu_0)^2 l D_c \snorm{\Omega_c}  + \calO(\varepsilon) \label{F1lem},
\end{align}
where we used the induction hypothesis and \eqref{eq:volume_fraction_ecm}. 

With the help of the above expansion for $F[\lambda_n^\varepsilon,\mu,l]$ 
we obtain
\begin{align}
	\lepsdummy(n+1) &= (F_0 + \varepsilon F_1^n -2 \mu l) \left(2 l - \dfrac{F_0 + \varepsilon F_1^n}{2\mu} \right)^{-1} + \calO(\varepsilon^2) \nonumber \\
&= 	(F_0 -2 \mu l) \left(2 l - \dfrac{F_0}{2\mu} \right)^{-1} + \varepsilon F_1^n l \left(2 l - \dfrac{F_0}{2\mu}\right)^{-2} \left(1 - \varepsilon \left(2 l - \dfrac{F_0}{2\mu}\right)^{-1} \dfrac{F_1^n}{2\mu}\right)^{-1} \nonumber \\
&\quad  + \calO(\varepsilon^2) \nonumber \\
&= 	\lambda_0 + \varepsilon F_1^n l \left(2 l - \dfrac{F_0}{2\mu}\right)^{-2} 
\sum_{k=0}^{\infty} \left( \varepsilon \left(2 l - \dfrac{F_0}{2\mu}\right)^{-1} \dfrac{F_1^n}{2\mu}\right)^{k}  + \calO(\varepsilon^2) \label{Fdecomp}
\end{align}
if $\varepsilon_0$ is chosen so small that 
\begin{align*}
\left|  \varepsilon \left(2 l - \dfrac{F_0}{2\mu}\right)^{-1} \dfrac{F_1^n}{2\mu}\right| &< 1.
\end{align*}
This implies
\begin{align}
\lepsdummy(n+1) &= \lambda_0 + \varepsilon F_1^n l \left(2 l - \dfrac{F_0}{2\mu}\right)^{-2} + \calO(\varepsilon^2).
\end{align}
Now, using \eqref{F0lem} and \eqref{F1lem} we get
\begin{align}
	\lepsdummy(n+1)	
	&= \lambdam + \varepsilon D_c \snorm{\Omega_c} + \calO(\varepsilon^2),
\end{align}
which yields the assertion of the lemma by induction.
\end{proof}

\begin{lem}\label{lem:induction_part1}
	Suppose that $\lepsdummy(1) \neq \lepsdummy(0)$. Then there exist $m \in \N$, $m \geq 2$, and $\varepsilon_0 > 0$ such that
	\begin{align*}
		\lepsdummy(1) = \lambdam + \varepsilon \snorm{\Omega_c} D_c +\varepsilon^m \lambda_m^1 + \calO(\varepsilon^{m+1})
	\end{align*}
	$\lambda^1_m \neq 0$ for all $\varepsilon \in (0,\varepsilon_0)$, where $\lambda^1_m \neq 0$.
\end{lem}
\begin{proof}
According to algorithm \ref{alg:emc_detail}, the initial value of the embedded cell algorithm is given by $\lepsdummy(0) = \lambdam + \varepsilon \snorm{\Omega_c} D_c$.
Then, according to theorem \ref{thm:generalization_power_series} the solution $u$ of \eqref{eq:weak_form_ECM} with the first Lamé parameter
\begin{align*}
	\lambda = \lambdam + \chi_{\widetilde{\Omega}_c} \varepsilon D_c + \chi_{\widetilde{\Omega}_\text{dummy}} \varepsilon \snorm{\Omega_c} D_c = \lambdam + \varepsilon \lpert
\end{align*}
is given by
\begin{align*}
	u = \sum_{k = 0}^\infty \varepsilon^k u_k,
\end{align*}
which is absolutely convergent for $\varepsilon < \varepsilon_0 \ll 1$ in $H^1(\Omega)$ and the $u_k$ satisfy the bound
\begin{align*}
	\norm[H^1(\Omega)]{u_k} = C^k \norm[L^\infty(\Omega)]{\lpert}^k \norm[H^1(\Omega)]{u_0}, \qquad C > 0.
\end{align*}
Hence, the tensile force has the representation
\begin{align*}
	F = \sum_{k = 0}^\infty \varepsilon^k F_k, \quad F_k = \int_\Omega \lpert \trace(\nabla^s u_{k-1}) + \lambdam \trace(\nabla^s u_k) dx.
\end{align*}
This yields
\begin{align*}
	\snorm{F_k} &\leq \snorm{\supp(\lpert)} \norm[L^\infty(\Omega)]{\lpert} \norm[H^1(\Omega)]{u_{k-1}} + \lambda_0 \norm[H^1(\Omega)]{u_k} \\
	&\leq \norm[L^\infty(\Omega)]{\lpert} \left(\snorm{\supp(\lpert)} + C \lambda_0\right) \norm[H^1(\Omega)]{u_{k-1}} 
\end{align*}
and hence, we have
\begin{align*}
	\sum_{k=0}^\infty \varepsilon^k \snorm{F_k} \leq \snorm{F_0} +  \norm[L^\infty(\Omega)]{\lpert} \left(\snorm{\supp(\lpert)} + C \lambda_0\right) \sum_{k=0}^\infty \norm[H^1(\Omega)]{\varepsilon^{k+1} u_k}.
\end{align*}
Therefore, the series is absolutely convergent for $\varepsilon \in (0, \varepsilon_0)$.

Now, proceeding analogously as in the derivation of \eqref{Fdecomp}, we obtain
\begin{align*}
	\lepsdummy(1) = \lambdam + \varepsilon \snorm{\Omega_c} D_c + \lambda_R
\end{align*}
with
\begin{align*}
	\lambda_R &= \varepsilon F_1 l \sum_{k=1}^\infty \left(\varepsilon \dfrac{F_1}{2\mu C_0}\right)^k + F_R l \left(\widetilde{C}^2 - \widetilde{C}\dfrac{F_R}{2\mu}\right)^{-1}, \\[2mm]
	F_R &= \sum_{k = 2}^\infty \varepsilon^k F_k, \quad
	\widetilde{C} = 2l - \dfrac{F_0 + \varepsilon F_1}{2\mu}, \\[2mm]
	\widetilde{C}^{-1} &= \left(2l - \dfrac{F_0 + \varepsilon F_1}{2\mu}\right)^{-1} = C_0^{-1} \sum_{k = 0}^\infty \left(\varepsilon \dfrac{F_1}{2\mu C_0}\right)^k, \quad
	C_0 = 2l - \dfrac{F_0}{2\mu}.
\end{align*}
This yields
\begin{align*}
	\lambda_R 
	&= \varepsilon F_1 l \sum_{k = 1}^\infty \left(\varepsilon \dfrac{F_1}{2\mu C_0}\right)^k + F_R l \left(C_0^{-1} \sum_{k=0}^\infty\left(\varepsilon \dfrac{F_1}{2\mu C_0}\right)^k\right)^{2} \sum_{k=0}^\infty \left(\dfrac{1}{2\mu}\left(\sum_{k=2}^\infty \varepsilon^k F_k \right) C_0^{-1}\left(\sum_{k=0}^\infty \left(\varepsilon \dfrac{F_1}{2\mu C_0}\right)^k\right)\right)^k.
\end{align*}
Note, that this expression contains only finitely many different series and hence, there exists an $\varepsilon_0 > 0$ such that all the occurring series converge absolutely for $\varepsilon \in (0,\varepsilon_0)$.
Therefore, by interchanging and multiplying the terms of the sums, the expression can be rewritten as
\begin{align*}
	\lambda_R = \sum_{k=2} \varepsilon^k \lambda^1_k,
\end{align*}
which converges absolutely.
Hence, we have
\begin{align*}
	\lepsdummy(1) = \lambdam + \varepsilon \snorm{\Omega_c} D_c + \sum_{k=2}^\infty \varepsilon^k \lambda_k^1
\end{align*}
for all $\varepsilon \in (0,\varepsilon_0)$.
Because of $\lepsdummy(1) \neq \lepsdummy(0)$ there exists a $m \in \N$, $m \geq 2$ such that $\lambda^1_m \neq 0$.
Hence, by choosing the smallest of these $m$, the assertion of the lemma follows.
\end{proof}

\begin{lem}\label{lem:induction_part2}
	Suppose that there exist $n \in \N$, $m \in \N$, $m \geq 2$, and $\varepsilon_0 > 0$ such that
	\begin{align*}
		\lepsdummy(j) = \sum_{k = 0}^{m-1} \varepsilon^k \lambda_k + \varepsilon^{m} \lambda^j_m + \calO(\varepsilon^{m+1})
	\end{align*}
	for $j = n, n+1$ and
	\begin{align}
		\lepsdummy(n+1) - \lepsdummy(n) = \varepsilon^m (\lambda^{n+1}_m - \lambda^n_m) + \calO(\varepsilon^{m+1}), \quad \lambda^{n+1}_m - \lambda^n_m \neq 0
		\label{eq:assertion}
	\end{align}
	for all $\varepsilon \in (0,\varepsilon_0)$.
	Then, for all $\widetilde{n} \geq n$ there exists an $\widetilde{\varepsilon}_0>0$ such that  
	\begin{align}
		\lepsdummy(\widetilde{n}) = \sum_{k=0}^{m-1} \varepsilon^k \lambda_k + \varepsilon^m \lambda^{\widetilde{n}}_m + \calO(\varepsilon^{m+1})
		\label{eq:expension}
	\end{align}
	and 
	\begin{align*}
		\lepsdummy(\widetilde{n}+1) - \lepsdummy(\widetilde{n}) = \varepsilon^m (\lambda^{\widetilde{n}+1}_m - \lambda^{\widetilde{n}}_m) + \calO(\varepsilon^{m+1}), \quad \lambda^{\widetilde{n}+1}_m - \lambda^{\widetilde{n}}_m \neq 0
	\end{align*}
	for all $\varepsilon \in (0,\widetilde{\varepsilon}_0)$.
	Furthermore, if there exists a $C > 0$ such that $\snorm{\lambda^{\widetilde{n}}_m} \leq C$ for all $\widetilde{n}$, then $\widetilde{\varepsilon}_0$ can be chosen independently of $\widetilde{n}$.
\end{lem} 
\begin{proof}
	Let be $j = n, n+1$.
	Then, analogously to the proof of lemma \ref{thm:monotonicity_force}, we obtain that the solution of \eqref{eq:weak_form_ECM} with the first Lamé parameter
	\begin{align*}
		\lambda^j = \chi_{\widetilde{\Omega}_m} \lambdam + \chi_{\widetilde{\Omega}_c} \lambdac^\varepsilon + \chi_{\widetilde{\Omega}_\text{dummy}} \lepsdummy(j)
	\end{align*}
	is given by
	\begin{align*}
		u^j = \sum_{k=0}^{m-1} \varepsilon^k u_k + \varepsilon^m u^j_m + \calO(\varepsilon^{m+1})
	\end{align*}
	and hence, the tensile force is given by
	\begin{align*}
		F^j = F[\lambda^j, \mu,l] = \underbrace{\sum_{k = 0}^{m-1} \varepsilon^k F_k}_{=: \widetilde{F}} + \varepsilon^m F^j_m + \calO(\varepsilon^{m+1}).
	\end{align*}
	Then, proceeding analogously as in the derivation of \eqref{Fdecomp}, we obtain that there exists an $\widetilde{\varepsilon}_0 >0$ (which can be chosen independently of $j$ if $F^j_m < \widetilde{C}$ for all $j$ and hence, if $\lambda^j_m < C$) such that for all $\varepsilon \in (0,\widetilde{\varepsilon}_0)$ it holds
	\begin{align*}
		\lepsdummy(j+1) &= \widetilde{\lambda} + \varepsilon^{m} F^j_m l \left(2l - \dfrac{\widetilde{F}}{2\mu}\right)^{-2} \sum_{k = 0}^\infty \left(\varepsilon \dfrac{F^j_m}{2\mu}\left(2l - \dfrac{\widetilde{F}}{2\mu}\right)^{-1}\right)^{k} + \calO(\varepsilon^{m+1}) \\
		&= \bar{\lambda} + \varepsilon^{m} F^j_m l \left(2l - \dfrac{F_0}{2\mu}\right)^{-2} + \calO(\varepsilon^{m+1}), 
	\end{align*}
where $\widetilde{\lambda}$ is determined by 
\begin{align*}
F[\widetilde{\lambda}, \mu, l] = \widetilde{F}
\end{align*}
and has an expansion of the form	
\begin{align*}
\widetilde{\lambda} = \sum_{k = 0}^\infty \varepsilon^k \widetilde{\lambda}_k
\end{align*}
with $\widetilde{\lambda}_k \in \R$. In particular, $\widetilde{\lambda}$ is independent of $j$.
By choosing $j = n$ we obtain, due to \eqref{eq:assertion}, that
	\begin{align*}
		\widetilde{\lambda} = \sum_{k=0}^{m-1} \varepsilon^k \lambda_k + \varepsilon^m \bar{\lambda}_m + \calO(\varepsilon^{m+1}).
	\end{align*}
	Therefore, by choosing $j = n+1$ we get that also $\lepsdummy(n+2)$ has an expansion of the form \eqref{eq:expension} and because of lemma \ref{thm:monotonicity_force} we have the implication
\begin{align*}
		\lambda_m^{n+1} \neq \lambda_m^n \Rightarrow F^{n+1}_m \neq F^{n}_m \Rightarrow  F^{n+1}_m l \left(2l - \dfrac{F_0}{2\mu}\right)^{-2} \neq F^n_m l \left(2l - \dfrac{F_0}{2\mu}\right)^{-2} \Rightarrow \lambda^{n+2}_m \neq \lambda^{n+1}_m,
	\end{align*}
	which yields the statement of the lemma by induction.
\end{proof}

Since the assertions of the lemmas \ref{thm:approximation_lambda_dummy}-\ref{lem:induction_part2} directly imply that $\lepsdummy(n)$ possesses a representation of the form \eqref{eq:representation_lepsdummy} which satisfies \eqref{neq}, the first step of the proof of \ref{lem:uniform_boundary_lambdaNM} is completed.

Now, we perform the second step of the proof. We prove by complete 
 induction using the iteration procedure from algorithm \ref{alg:emc_detail} that $\lambda_R^n:= \lambda_{m,n} + \varepsilon R_{\varepsilon,n}$ can be bounded independently of $n$ and $\varepsilon$ if $\varepsilon < \varepsilon_0$ for some $\varepsilon_0$ independent of $n$.
	Therefore, we consider for an arbitrary $n \in \N$:
	\begin{align*}
		\lepsdummy(n) = \lambdam + \varepsilon \snorm{\Omega_c} D_c + \varepsilon^{m} \lambda_R^n.  
	\end{align*}
	Then, using the ansatz
	\begin{align*}
		u = \sum_{k = 0}^{m-1} \varepsilon^k u^n_k + \varepsilon^m u_R^n
	\end{align*}
	for the solution of the tensile test with first Lamé parameter
	\begin{align*}
		\lambda^{\varepsilon,n} := \lambdam + \varepsilon (\chi_{\widetilde{\Omega}_c} + \chi_{\widetilde{\Omega}_\text{dummy}} \snorm{\Omega_c}) D_c + \varepsilon^m \chi_{\widetilde{\Omega}_\text{dummy}} \lambda^n_R =: \lambdam + \varepsilon \lpert + \varepsilon^m \lpert^R
	\end{align*}
	yields
	\begin{align*}
		&\int_\Omega \left[\lambdam \trace(\nabla^s u^n_0) I + 2 \mu \nabla^s u^n_0\right] : \nabla^s v dx = 0 \\
		&\int_\Omega \left[\lambdam \trace(\nabla^s u^n_k) l + 2\mu \nabla^s u^n_k\right] : \nabla^s v dx = - \int_\Omega \lpert \trace(\nabla^s u^n_{k-1}) \trace(\nabla^s v) dx, \quad k = 1,\dots,m-1 \\
		&\int_\Omega \left[\lambda^{\varepsilon,n} \trace(\nabla^s u_R^n) I + 2\mu \nabla^s u_R^n\right] : \nabla^s v dx \\
		&\pushright{= - \int_\Omega \left[\lpert^R \sum_{k=0}^{m-1} \varepsilon^k \trace(\nabla^s u^n_k) + \lpert \trace(\nabla^s u^n_{m-1})\right] \trace(\nabla^s v) dx}
	\end{align*}
	for all $v \in \calW_0$. Consequently, the functions $u^n_k$ are independent of $n$ for $k = 0, \dots m-1$ (and will, from now on, be denoted by $u_k$) and
		there exist constants $C_k > 0$ independently of $n$ such that
		\begin{align*}
			\norm[H^1(\Omega)]{u_k} \leq C_k.
		\end{align*}
	Hence, there exists a $C > 0$ independently of $n, \varepsilon$ such that a bound for $u_R^n$ is given by
	\begin{align*}
		\norm[H^1(\Omega)]{u_R^n} &\leq C \left[\norm[L^\infty(\Omega)]{\lpert^R} \sum_{k = 0}^{m-1} \varepsilon^k \norm[H^1(\Omega)]{u_k} + \norm[L^\infty(\Omega)]{\lpert} \norm[H^1(\Omega)]{u_{m-1}}\right] \\
		&\leq C \left[\norm[L^\infty(\Omega)]{\lpert^R} \left((1-\nu_0) l + \sum_{k = 1}^{m-1} \varepsilon^k C_k\right) + \norm[L^\infty(\Omega)]{\lpert} C_{m-1}\right] \\
		&= \widetilde{C}_1(\varepsilon, l) \snorm{\lambda_R^n} + \widetilde{C}_2
	\end{align*}
	with $\widetilde{C}_2 > 0$ independently of $n ,\varepsilon$ for $\varepsilon < 1$ and 
	\begin{align*}
		\widetilde{C}_1(\varepsilon, l) = C\left((1-\nu_0) l + \sum_{k = 1}^{m-1} \varepsilon^k C_k\right).
	\end{align*}

Then, the tensile force is given by
	\begin{align*}
		F[\lambda^{\varepsilon,n}, \mu, l] &= \sum_{k=0}^{m-1} \varepsilon^k F_k + \varepsilon^m F_R^n,
\end{align*}
where		
\begin{align*}		
		F^n_R &= \int_\Omega \lpert^R \sum_{k = 0}^{m-1} \varepsilon^k \trace(\nabla^s u_k) + \lpert \trace(\nabla^s u_{m-1}) + \lambda^{\varepsilon,n} \trace(\nabla^s u_R^n) dx.
	\end{align*}
	By using the estimate for $u_R^n$ in $H^1$ we obtain
	\begin{align*}
		\snorm{F_R^n} &\leq \snorm{\lambda_R^n} \sum_{k=0}^{m-1} \varepsilon^k C_k + \norm[L^\infty(\Omega)]{\lpert} C_{m-1} + \norm[L^\infty(\Omega)]{\lambda^{\varepsilon,n}} \norm[H^1(\Omega)]{u_R^n} \\
		&\leq \widetilde{\widetilde{C}}_1(\varepsilon,l) \snorm{\lambda_R^n} + \varepsilon^{m}\widetilde{\widetilde{C}}_2(\varepsilon,l) \snorm{\lambda_R^n}^2 +\widetilde{\widetilde{C}}_3
	\end{align*}
	with 
	\begin{align*}
		\widetilde{\widetilde{C}}_1(\varepsilon,l) &=  C(1+\lambdam + \varepsilon D_c)  \Big((1- \nu_0) l + \sum_{k=1}^{m-1} \varepsilon^k C_k\Big)+ \varepsilon^{m} \widetilde{C}_2,\\
\widetilde{\widetilde{C}}_2(\varepsilon,l)&={\widetilde{C}}_1(\varepsilon,l)	
\end{align*}	
and $\widetilde{\widetilde{C}}_3 > 0$ independently of $n, \varepsilon$ for $\varepsilon < 1$.	
	
Then, proceeding analogously as in the derivation of \eqref{Fdecomp}, we get
	\begin{align*}
		\lepsdummy(n+1) = \widetilde{\lambda} + \varepsilon^m F_R^n \left(K^2 - \varepsilon^m \dfrac{K F_R^n}{2\mu}\right)^{-1},
	\end{align*}
	where
	\begin{align*}
		K = 2l - \frac{1}{2\mu}\Big( \sum_{k = 0}^{m-1} \varepsilon^k F_k \Big)
	\end{align*}
	and $\widetilde{\lambda} = \lambdam + \varepsilon \snorm{\Omega_c} D_c + \varepsilon^m \widetilde{\lambda}_R$ with $ \widetilde{\lambda}_R \in \R$ independent of $n$.
	Then, by using that $\lepsdummy(n+1)$ also has the representation \eqref{eq:representation_lepsdummy}, we get 
	\begin{align}
		\snorm{\lambda_R^{n+1}} 
		&\leq \snorm{\widetilde{\lambda}_R} + \snormlr{F_R^n} \snormlr{K^2 - \varepsilon^m \dfrac{\snormlr{K} \snormlr{F^n_R}}{2\mu}}^{-1}.
		\label{eq:estimate0}
	\end{align}
	
To examine the right-hand side in detail, this motivates the definition of a function $\Phi$ with
\begin{align*}
	\Phi(x) := \snorm{\widetilde{\lambda}_R} + \snorm{f(x)} \snormlr{K^2 - \varepsilon^m \dfrac{\snorm{K} \snorm{f(x)}}{2\mu}}^{-1},
\end{align*}
where
\begin{align*}
	\snorm{f(x)} \leq \widetilde{\widetilde{C}}_1(\varepsilon,l) \snorm{x} + \varepsilon^{m}\widetilde{\widetilde{C}}_2(\varepsilon,l) \snorm{x}^2  + \widetilde{\widetilde{C}}_3.
\end{align*}
Now, the aim is to prove that there exists an $r>0$ such that $\Phi$ maps the ball around zero with radius $r$ onto itself.
For all $r > 0$ there exists $\widetilde{\varepsilon}_0(r)$ such that
\begin{align}
	0 \leq \varepsilon^m \dfrac{\snorm{K} \snorm{f(x)}}{2\mu}  \leq \dfrac{K^2}{2}
	\label{eq:estimate2}
\end{align}
for all $x \in B_r(0)$ and $0 < \varepsilon \leq \widetilde{\varepsilon}_0(r)$.
This yields
\begin{align*}
	\Phi(x) \leq \snorm{\widetilde{\lambda}_R} + \dfrac{2}{K^2} (\widetilde{\widetilde{C}}_1(\varepsilon,l) r +  \varepsilon^{m}\widetilde{\widetilde{C}}_2(\varepsilon,l) r^2   + \widetilde{\widetilde{C}}_3)
\end{align*}
for all $x \in B_r(0)$.
Hence, in order to gain the desired property of $\Phi$, it has to be ensured that
\begin{align*}
	\snorm{\widetilde{\lambda}_R} + \dfrac{2}{K^2}(\widetilde{\widetilde{C}}_1(\varepsilon,l) r + \varepsilon^{m}\widetilde{\widetilde{C}}_2(\varepsilon,l) r^2  + \widetilde{\widetilde{C}}_3) &\leq r ,
\end{align*}
which is equivalent to
\begin{align}
	& r \geq \left(\snorm{\widetilde{\lambda}_R} + \dfrac{2 \widetilde{\widetilde{C}}_3}{K^2}\right) \left(1 - \dfrac{2 (\widetilde{\widetilde{C}}_1(\varepsilon,l)+ \varepsilon^{m}\widetilde{\widetilde{C}}_2(\varepsilon,l) r)}{K^2}\right)^{-1}.
	\label{eq:estimate3}
\end{align}
Now, let $r_0 > 0$ such that
\begin{align*}
	r_0 \geq 2\left(\snorm{\widetilde{\lambda}_R} + \dfrac{2 \widetilde{\widetilde{C}}_3}{\underline{K}^2}\right),
\end{align*}
where $\underline{K}^2 > 0$ is a lower bound for $K^2$, which is independent of $\varepsilon$ for $\varepsilon < 1$.
Since $\widetilde{\widetilde{C}}_1$ is monotonically increasing with respect to $\varepsilon$ and $l$ and $\widetilde{\widetilde{C}}_1 \rightarrow 0$ for $(\varepsilon,l) \rightarrow 0$,
there exist $\varepsilon_0 \in (0, \widetilde{\varepsilon}_0(r_0))$ and $l_0 >0$ such that
$r=r_0$ satisfies \eqref{eq:estimate3} for all $0 < \varepsilon < {\varepsilon}_0$ and all $0<l<l_0$.

Due to \eqref{eq:estimate0}, this yields
\begin{align*}
\snorm{\lambda_R^{n}} \leq r_0 \;\, \Longrightarrow	\;\, \snorm{\lambda_R^{n+1}} \leq r_0  
\end{align*}
for all $0 < \varepsilon < \varepsilon_0$, where $\varepsilon_0$ does not depend on $n$.
Using $\lambda_R^0 = 0 \in B_{r_0}(0)$ we obtain by induction that $\lambda_R^{n} \in B_{r_0}(0)$ for all $n \in  \N$ and all $0 < \varepsilon < \varepsilon_0$.
\end{proof}

\textbf{Remark:} The bound for the prescribed deformation length $l$ at the top end, which was used to prove lemma \ref{lem:uniform_boundary_lambdaNM}, does not impose a further restriction to the convergence result of the embedded cell method because the result of algorithm \ref{alg:emc_detail} is independent of $l$ (cf. lemma \ref{lem:lambda_equiv_independent_of_l}).

Now, using the above lemmas \ref{thm:monotonicity_force} and  \ref{lem:uniform_boundary_lambdaNM}, we can prove the convergence of the embedded cell method. 

\begin{theorem}[convergence of the embedded cell method]\label{thm:convergence_ECM}
	Suppose, without loss of gen\-erality that $\lambdam \leq \lambdac^\varepsilon$.
	Furthermore, let $(\lepsdummy(n))_{n\in \N}$ be the sequence of dummy material parameters given by algorithm \ref{alg:emc_detail}.
	Then, there exists an $\varepsilon_0 > 0$ such that the sequence $(\lepsdummy(n))_{n\in \N}$ converges monotonically to some $\lepsdummy \in [\lambdam, \lambdac^\varepsilon]$ for all $\varepsilon \in (0, \varepsilon_0)$.
\end{theorem}
\begin{proof}
If $\lambdam = \lambdac^\varepsilon$, then $\lepsdummy(0) = \lambda_m$ and as the material considered as embedded cell is homogeneous this yields $\lepsdummy(n) = \lambdam$ for all $n \in \N$ and $\varepsilon > 0$.
	Moreover, the cases $\snorm{\Omega_c} \in \{0,1\}$ yield that $\lepsdummy(n) = \lambdam$ and $\lepsdummy(n) = \lambdac^\varepsilon$, respectively, with the same argument as above.
	
	Therefore, let be $\lambdam < \lambdac^\varepsilon$ and $0 < \snorm{\Omega_c} < 1$.
	Because of lemma \ref{lem:uniform_boundary_lambdaNM} there exists $\tilde{\varepsilon}_0 > 0$ such that
	\begin{align*}
		\lepsdummy(n) = \lambdam + \varepsilon \snorm{\Omega} D_c + \calO(\varepsilon^2)
	\end{align*}
	uniformly for all $n \in \N$ and all $\varepsilon \in (0,\tilde{\varepsilon}_0)$. 
	Hence, with $D_c > 0$ according to the assumption there exists an $\varepsilon_0 > 0$ such that
	\begin{align*}
		\lepsdummy(n) - \lambdam &= \varepsilon \snorm{\Omega_c} D_c + \calO(\varepsilon^2) > 0 ,\\
		\lambdac^\varepsilon - \lepsdummy(n) &= \varepsilon (1- \snorm{\Omega_c}) D_c + \calO(\varepsilon^2) > 0
	\end{align*}
	for all $\varepsilon \in (0,\varepsilon_0)$.

Therefore, $\lepsdummy(n) \in [\lambdam, \lambdac^\varepsilon]$ for all $n \in \N$, which also yields that if the sequence converges to some $\lepsdummy \in \R$, this implies $\lepsdummy \in [\lambdam, \lambdac^\varepsilon]$.
Hence, it remains to prove that the sequence is monotone since a monotone and bounded sequence in $\R$ is convergent.

First, consider the case when $\lepsdummy(0) = \lepsdummy(1)$.
This means that the initial value is a fixed point of the iteration procedure and hence, this implies that $\lepsdummy(n) = \lepsdummy(0)$.
Hence, the sequence is constant and obviously monotonically convergent.

Now, let $\lepsdummy(0) < \lepsdummy(1)$. 
Because of lemma \ref{lem:uniform_boundary_lambdaNM}, this yields that there exist $m \in \N$, $m \geq 2$, and $\varepsilon_0 > 0$ such that
\begin{align*}
	\lepsdummy(1) = \lepsdummy(0) + \varepsilon^m \lambda_{m,1} + \calO(\varepsilon^{m+1}), \quad \lambda_{m,1} > 0
\end{align*}
for all $\varepsilon \in (0,\varepsilon_0)$.
Let
\begin{align*}
	\lambda^{\varepsilon,1} := \lambdam + \varepsilon \chi_{\widetilde{\Omega}_c}  D_c + \chi_{\widetilde{\Omega}_\text{dummy}} (\lepsdummy(1) -\lambdam).
\end{align*}
Then, lemma \ref{thm:monotonicity_force} can be applied and yields
\begin{align*}
	F[\lambda^0,\mu,l] < F[\lambda^{\varepsilon,1},\mu,l].
\end{align*}
Moreover, the mapping $\phi : \R \rightarrow \R$ defined by
\begin{align*}
	\phi(x) = \left(x - 2\mu l\right) \left(2l - \dfrac{x}{2\mu}\right)^{-1},
\end{align*}
which defines the equivalent first Lamé parameter, is monotone.
Hence, we get
\begin{align*}
	\lepsdummy(1) < \lepsdummy(2).
\end{align*}
Using the same arguments we obtain by induction that the sequence of dummy material parameters is monotonically increasing for all $\varepsilon \in (0, \varepsilon_0)$.

Finally, for $\lepsdummy(0) > \lepsdummy(1)$ we get analogously that the sequence of dummy material parameters is monotonically decreasing for all $\varepsilon \in (0, \varepsilon_0)$. 
\end{proof}

Theorem \ref{thm:convergence_ECM} yields the existence of a limit dummy parameter $\lepsdummy$.
The final aim of this section is to show that this limit admits an expansion of the form \eqref{eq:approximation_dummy_parameter}.
If such a result holds, then the embedded cell method \ref{alg:emc_detail} yields correct results up to errors of second order with respect to $\varepsilon$.

\begin{theorem}[justification of the embedded cell method]
\label{thm:correctness}
Let $\lambdam, \lambdac^\varepsilon := \lambdam + \varepsilon D_c$ for $\varepsilon > 0$ be the first Lamé parameters of a metal and a ceramic material, respectively.
Furthermore, suppose that these materials have the same shear modulus $\mu > 0$.
Then there exists an $\varepsilon_0 > 0$ such that the embedded cell methods defined by algorithm \ref{alg:emc_detail} converges to 
\begin{align}
	\lepsdummy = \lambdam + \varepsilon \snorm{\Omega_c} D_c + \calO(\varepsilon^2)
\end{align}
for all $\varepsilon \in (0, \varepsilon_0)$.
Moreover, if $\lambda^{\varepsilon,\text{hom}}$ is the effective material parameter given by theorem \ref{thm:approximation_result}, then 
\begin{align}
	\lambda^{\varepsilon,\text{hom}} - \lepsdummy = \calO(\varepsilon^2)
\end{align}
and consequently
\begin{align}
	F[\lambda^{\varepsilon,\text{hom}},\mu,l] - F[\lepsdummy,\mu,l] = \calO(\varepsilon^2)
\end{align}
for all $\varepsilon \in (0, \varepsilon_0)$.
\end{theorem}
\begin{proof}
The existence of a limit $\lepsdummy$ is given by theorem \ref{thm:convergence_ECM}.
Furthermore, according to lemma \ref{lem:uniform_boundary_lambdaNM}, there exists an $\varepsilon_0 > 0$, a $m \in \N$, $m \geq 2$ and a constant $C > 0$ independently of $n$ such that
\begin{align*}
	\lepsdummy(n) = \lambdam + \varepsilon \snorm{\Omega_c} D_c + \varepsilon^m \lambda_R^n
\end{align*}
for all $\varepsilon \in (0, \varepsilon_0)$, and $\snorm{\lambda_R^n} \leq C$ for all $n \in \N$.
Since the sequence converges, this yields 
\begin{align*}
	\lepsdummy = \lambdam + \varepsilon \snorm{\Omega_c} D_c + \varepsilon^m \lambda_R
\end{align*}
with $\snorm{\lambda_R} \leq C$, which implies the first assertion of the theorem.
Finally, according to theorem \ref{thm:approximation_result}, we have
\begin{align*}
	\lambda^{\varepsilon,\text{hom}} = \lambdam + \varepsilon \snorm{\Omega_c} D_c + \calO(\varepsilon^2)
\end{align*}
for all $\varepsilon \in (0, \varepsilon_0)$, which directly implies the second assertion.
\end{proof}

\appendix
\section{Technical results}
The appendix contains some technical results needed in the previous sections but whose proofs are rather technical and are not directly contributing to the understanding of the topic.

\subsection{Well-posedness of the embedded cell method}\label{chap:appendix_well-posedness}
In this subsection the following two results will be proven.
The first one is that there exists an equivalent material parameter $\lambda^\text{equiv} \in \R^+$ under suitable assumption on the first Lamé parameter function $\lambda(x)$ of the material.
Furthermore, the second result is that the equivalent material parameter is independent of the choice of the tensile length $l$.
These results also imply that the result of the embedded cell algorithm \ref{alg:emc_detail} exists and is independent of $l$ and thus, the algorithm is well-posed at least for small perturbations.

\begin{lem}\label{lem:existence_equiv_param}
	Let $\varepsilon > 0$, $\lambda^\varepsilon \in L^\infty(\Omega)$ with $\lambda^\varepsilon = \lambda_0 + \varepsilon \lpert \geq 0$, $\mu \equiv \mu_0 > 0$ and $l \in \R$.
Then there exists an $\varepsilon_0 > 0$ such that the equivalent material parameter $\lambda^\text{equiv} \in \R$ according to definition \ref{def:equivalent_material_param} is non-negative for all $\varepsilon \in [0,\varepsilon_0)$.
\end{lem}
\begin{proof}
	Let w.l.o.g be $l > 0$ and $\lpert \geq 0$ almost everywhere.
	Furthermore, let $u^\varepsilon \in \calW_l$ be the weak solution of the tensile test. 
	Then, according to theorem \ref{thm:approximation_result_u} $u^\varepsilon$ is given by
	\begin{align*}
		u^\varepsilon = u_0 + \varepsilon u_1 + \calO(\varepsilon^2).
	\end{align*}
	By analogous arguments as in the proof of lemma \ref{thm:approximation_lambda_dummy}
	we obtain that the tensile force $F[\lambda^\varepsilon, \mu,l]$ is given by
	\begin{align*}
		F[\lambda^\varepsilon, \mu,l] &= F_0 + \varepsilon F_1 + \calO(\varepsilon^2)
	\end{align*}
	with	
	\begin{align*}
		F_0 &= (1-\nu_0)l \lambda_0 + 2\mu l > 0, \\
		F_1 &= (1-\nu_0)^2l \int_\Omega \lpert dx \geq 0
	\end{align*}
	and $\lambda^\text{equiv}$ satisfies
	\begin{align*}
		\lambda^\text{equiv} &= (F_0 + \varepsilon F_1 - 2\mu l) \left(2l - \dfrac{F_0 + \varepsilon F_1}{2\mu}\right)^{-1} + \calO(\varepsilon^2) \\
		&=  \lambda_0 + \varepsilon \int_\Omega \lpert dx  + \calO(\varepsilon^2) \\
		&\geq  0
	\end{align*}
if $\varepsilon_0 > 0$ is sufficiently small, which proves the lemma.
\end{proof}

\begin{lem}\label{lem:lambda_equiv_independent_of_l}
	Let $\varepsilon > 0$, $\lambda^\varepsilon \in L^\infty(\Omega)$ with $\lambda^\varepsilon = \lambda_0 + \varepsilon \lpert \geq 0$, $\mu \equiv \mu_0 > 0$, $l \in \R$ and $\lambda^\text{equiv}$ be the equivalent material parameter according to definition \ref{def:equivalent_material_param}.
	Then, $\lambda^\text{equiv}$ does not depend on $l$.
\end{lem}
\begin{proof}
Let $L^{-1} : \R \rightarrow H_m^1(\Omega) \times H^1(\Omega)$ be the operator which assigns the solution of the tensile test equation to a given prescribed tensile length.
	Then, according to theorem \ref{thm:existence_uniqueness}, $L^{-1}$ is linear and injective and consequently invertible on its range with inverse $L$, which is also a linear operator.
	Furthermore, the mapping $F_u : \calR(L^{-1}) \rightarrow \R$ with
	\begin{align*}
		u \mapsto \int_\Omega \lambda \trace(\nabla^s u) + 2\mu \partial_2 u_2 dx
	\end{align*}
	is linear for $u \in \calR(L^{-1})$.
	Therefore, the mapping $F_l : \R \rightarrow \R$ with $F_l = F_u \circ L$ is a linear mapping.
	Hence, there exists an $a \in \R$ independent of $l$ such that $F_l = al$.
	This yields that
	\begin{align*}
		\lambda^\text{equiv} = (F_l - 2\mu l) \left(2l - \dfrac{F_l}{2\mu}\right)^{-1} = (a - 2\mu) \left(2 - \dfrac{a}{2\mu}\right)^{-1},
	\end{align*}
	which is independent of $l$.
\end{proof}

\subsection{Generalized approximation results}\label{sec:generalized_approximation_results}
In this subsection, two generalized approximation results will be proven. 
The first one is a generalization of theorem \ref{thm:approximation_result_u} for higher error orders.
The second one states that the solution can be represented as a power series in $\varepsilon$.

\begin{theorem}\label{thm:generalized_approximation_u}
	Let $0 < \varepsilon <1$, $\mu \in \R$, $m \in \N$ and $\lambda^\varepsilon$ be given by
	\begin{align*}
		\lambda^\varepsilon = \sum_{k = 0}^{m} \varepsilon^k\lambda_k + \calO(\varepsilon^{m+1})
	\end{align*}
	with $\lambda_0 \in \R$ and $\lambda_k \in L^\infty(\Omega)$ for $k = 0, \dots, m$, independent of $\varepsilon$.
	Then there exist functions $u_k, k = 0, \dots, m$ with $u_0 \in \calW_l$ and $u_k \in\calW_0$ for $k \geq 1$ such that for the solution $u^\varepsilon \in \calW_l$ of 
	\begin{align}
		\int_\Omega \lambda^\varepsilon \trace(\nabla^s u^\varepsilon) \trace(\nabla^s v) + 2\mu \nabla^s u^\varepsilon : \nabla^s v dx = 0
		\label{eq:equation_tensile_test}
	\end{align}
	for all $v \in \calW_0$ satisfies
	\begin{align}
		u^\varepsilon - \sum_{k = 0}^m \varepsilon^k u_k = \calO(\varepsilon^{m+1})
       \label{estRext}	
	\end{align}
	with respect to the $H^1$-norm.
\end{theorem}
\begin{proof}
Let $u_0 \in \calW_l$ and $u_k \in \calW_0$, $k = 1,\dots, m$, respectively, the unique weak solutions of 
\begin{align*}
	&\int_\Omega \left[\lambda_0 \trace(\nabla^s u_0) I + 2\mu_0 \nabla^s u_0\right] : \nabla^s v dx = 0 ,\\
	&\int_\Omega \left[\lambda_0 \trace(\nabla^s u_k) I + 2\mu_0 \nabla^s u_k\right] : \nabla^s v dx = - \sum_{i=0}^{m-1} \int_\Omega \lambda_{m-i} \trace(\nabla^s u_i) \trace(\nabla^s v) dx 
\end{align*}
for all $v \in \calW_0$.
Notice that the solution $u_0$ is given by \eqref{eq:u_0}.
By theorem \ref{thm:existence_uniqueness} we have for $u_k$, $k = 1, \dots, m$:
\begin{align*}
	\norm[H^1(\Omega)]{u_k} &\leq C \norm[\calW_0']{\sum_{i = 0}^{m-1} \int_\Omega \lambda_{m-i} \trace(\nabla^s u_i) \trace(\nabla^s v) dx} \\
	&\leq C \sum_{i = 0}^{m-1} \norm[L^\infty(\Omega)]{\lambda_{m-i}} \norm[H^1(\Omega)]{u_i}.
\end{align*}
In particular, the $\lambda_k$ are independent of $\varepsilon$.
Hence, there exists constants $C_k > 0$ for $k = 0, \dots, m$ independent of $\varepsilon$ such that
\begin{align}
	\norm[H^1(\Omega)]{u_k} \leq C_k.
	\label{eq:estimates_uk_appendix}
\end{align}

Now, define
\begin{align}
	\uapprox^\varepsilon &:= \sum_{k = 0}^m \varepsilon^k u_k,
	\label{eq:ansatz_appendix}
\end{align}
which is an element of $\calW_l$. Then $\uapprox^\varepsilon$ satisfies
\begin{align*}
 	\int_\Omega \lambda^\varepsilon \trace(\nabla^s \uapprox^\varepsilon) \trace(\nabla^s v)   + 2\mu \nabla^s \uapprox^\varepsilon : \nabla^s v dx &= -\widetilde{F}(v)
\end{align*} 
for all $v \in \calW_0$,	
where 	
 \begin{align*}	
 	\widetilde{F}(v) &=  \sum_{i + j > m} \int_\Omega \varepsilon^{i+j}\lambda_i \trace(\nabla^s u_j) \trace(\nabla^s v) dx.
\end{align*}
Therefore, $R := u^\varepsilon - \uapprox^\varepsilon$ solves
\begin{align*}
	\int_\Omega \lambda^\varepsilon \trace(\nabla^s R) \trace(\nabla^s v)   + 2\mu \nabla^s R : \nabla^s v = \widetilde{F}(v)
\end{align*}
for all $v \in \calW_0$ and we can use theorem \ref{thm:existence_uniqueness} again to obtain
\begin{align*}
\norm[H^1(\Omega)]{R} &\leq	C \norm[\calW_0']{\widetilde{F}} \\
&= C \sup_{\substack{v \in \calW_0, \\\norm[H^1(\Omega)]{v} = 1}} \snormlr{\sum_{i + j > m} \int_\Omega \varepsilon^{i+j}\lambda_i \trace(\nabla^s u_j) \trace(\nabla^s v) dx} \\
	&\leq \varepsilon^{m+1} C \sum_{i+j > m} \varepsilon^{i+j-m-1} \norm[L^\infty(\Omega)]{\lambda_i} \norm[H^1(\Omega)]{u_j} \\
	&\leq \varepsilon^{m+1} C \sum_{i+j > m} \varepsilon^{i+j-m-1} \norm[L^\infty(\Omega)]{\lambda_i} C_j \\
&\leq \varepsilon^{m+1} \widetilde{C}	
\end{align*}
for constants $C,\widetilde{C} >0$, independent of $\varepsilon$, which yields \eqref{estRext}.
\end{proof}

\begin{theorem}\label{thm:generalization_power_series}
	There exists an $\varepsilon_0 > 0$ such that the solution $u^\varepsilon \in \calW_l$ of 
	\begin{align}
		\int_\Omega \lambda^\varepsilon \trace(\nabla^s u^\varepsilon) \trace(\nabla^s v) + 2\mu \nabla^s u^\varepsilon : \nabla^s v dx = 0
		\label{eq:equation_appendix_series}
	\end{align}
	for all $v \in \calW_0$, where $0 < \varepsilon \leq \varepsilon_0$,  $\lambda^\varepsilon = \lambda_0 + \varepsilon \lpert$, $\lpert \in L^\infty(\Omega)$ and $\mu \in \R^+$, has  the representation
	\begin{align}
		u^\varepsilon = \sum_{k=0}^\infty \varepsilon^k u_k,
		\label{eq:ansatz_appendix_series}
	\end{align}
	which converges absolutely in $H^1$ for $\varepsilon \in [0, \varepsilon_0)$, $u_0 \in \calW_l$ and $u_k \in \calW_0$ for $k \geq 1$.
	Furthermore, there exists a constant $C > 0$ independently of $k, \varepsilon$ such that
	\begin{align*}
		\norm[H^1(\Omega)]{u_k} \leq C^k \norm[L^\infty(\Omega)]{\lpert}^k \norm[H^1(\Omega)]{u_0}
	\end{align*}
	for all $k \geq 0$.
\end{theorem}
\begin{proof}
Inserting the ansatz \eqref{eq:ansatz_appendix_series} into \eqref{eq:equation_appendix_series} yields
\begin{align*}
	&\int_\Omega \left[\lambda_0 \trace(\nabla^s u_0) I + 2 \mu \nabla^s u_0\right] : \nabla^s v dx = 0 ,\\
	&\int_\Omega \left[\lambda_0 \trace(\nabla^s u_k) I + 2 \mu \nabla^s u_k\right] : \nabla^s v dx = - \int_\Omega \lpert \trace(\nabla^s u_{k-1}) \trace(\nabla^s v) dx
\end{align*}
for $k \geq 1$ and all $v \in \calW_0$.
Notice that $u_0 \in \calW_0$ is given by \eqref{eq:u_0}.
By theorem \ref{thm:existence_uniqueness} there exists a $C > 0$ being independent of $\varepsilon > 0$ such that
\begin{align*}
	\norm[H^1(\Omega)]{u_k} \leq C \norm[L^\infty(\Omega)]{\lpert} \norm[H^1(\Omega)]{u_{k-1}}.
\end{align*}
By induction we get
\begin{align*}
\|u_k\|_{H^1(\Omega)} \leq C^k \|{\lpert}\|_{L^\infty(\Omega)}^k \|u_0\|_{H^1(\Omega)}.
\end{align*}
Hence, the series \eqref{eq:ansatz_appendix_series} converges absolutely for
\begin{align*}
\varepsilon < \dfrac{1}{C \|{\lpert} \|_{L^\infty(\Omega)}^k}
\end{align*}
and since $H^1(\Omega)$ is a Banach space the series also converges in $H^1$.
Because the left-hand-side of the equation \eqref{eq:equation_appendix_series} defines a continuous bilinear form in $H^1$ it follows that \eqref{eq:ansatz_appendix_series} is a solution of \eqref{eq:equation_appendix_series} and hence the assertion of the theorem follows.
\end{proof}


\end{document}